%% file: estimator_trimming.tex
\documentclass{article}

\usepackage{graphicx}
\usepackage{amssymb}
\usepackage[pdftex]{hyperref}
\usepackage{xcolor}
\hypersetup{
	colorlinks,
	linkcolor={red!50!black},
	citecolor={red!50!black},
	urlcolor={red!50!black}
}
\usepackage{cite}
\usepackage[hmargin={2.5cm,2.5cm}, vmargin={3cm,3cm}, dvips]{geometry}

\usepackage{fancyhdr}
\newcommand{\spn}[1]{\mathrm{span}\left\{#1\right\}}
\newcommand{\supp}[1]{\mathrm{supp}\left(#1\right)}

\usepackage{amsmath} 
\usepackage{amssymb} 
\usepackage{amsthm} 
\usepackage{array}
\usepackage{placeins}
\usepackage{enumerate}
\usepackage{esint}
\usepackage{graphics}
\usepackage{graphicx}
\usepackage{tikz}
\usetikzlibrary{patterns}
\usetikzlibrary{arrows}
\usetikzlibrary{calc}

\usetikzlibrary{shapes,positioning,shadows,trees,bending}
\usetikzlibrary{arrows.meta,chains}

\usepackage{pgfplots}
\pgfplotsset{compat=1.16}
\usepackage{caption}
\usepackage{subcaption}
\usepackage{listings}
\usepackage[shortlabels]{enumitem}
\makeatletter
\def\namedlabel#1#2{\begingroup
	#2%
	\def\@currentlabel{#2}%
	\phantomsection\label{#1}\endgroup
}
\usepackage{multirow}
\usepackage{multicol}
\usepackage{wrapfig}

\usepackage{afterpage}
\usepackage{mathrsfs}
\usepackage{mathtools}

\newtheoremstyle{droit}
{}
{}
{\upshape}
{}
{\bfseries}
{}
{ }
{}
\newtheoremstyle{italique}
{}
{}
{\itshape}
{}
{\bfseries}
{}
{ }
{}
\theoremstyle{italique}
\newtheorem{theorem}{Theorem}[section]

\newtheorem{lemma}[theorem]{Lemma}

\theoremstyle{droit}

\newtheorem{remark}[theorem]{Remark}

\newtheorem{definition}[theorem]{Definition}

\usepackage{algorithm,algcompatible}
\usepackage{algpseudocode}
\algnewcommand{\lst}{\texttt{lst}}
\algnewcommand{\slst}{\texttt{slst}}
\algnewcommand{\SEND}{\textbf{send}}
\newsavebox{\algleft}
\newsavebox{\algright}

\usepackage{geometry}
\usetikzlibrary{shapes,backgrounds}

\newcommand{\review}[1]{{\color{black}#1}}

\newcommand{\trim}{\text{cut}}
\newcommand{\untr}{\text{act}}
\newtheorem{mylemma}{Appendix}[section]

\definecolor{darkgreen}{rgb}{0,0.4,0} 
\definecolor{darkbrown}{rgb}{0.5, 0.396, 0.09}

\definecolor{c1}{rgb}{0.0, 0.4196078431372549, 0.6431372549019608}
\definecolor{c2}{rgb}{1.0, 0.5019607843137255, 0.054901960784313725}
\definecolor{c3}{rgb}{0.6705882352941176, 0.6705882352941176,
	0.6705882352941176} \definecolor{c}{rgb}{0.34901960784313724, 0.34901960784313724, 0.34901960784313724}
\definecolor{c4}{rgb}{0.37254901960784315, 0.6196078431372549,
	0.8196078431372549} \definecolor{c}{rgb}{0.7843137254901961, 0.3215686274509804, 0.0}
\definecolor{c5}{rgb}{0.5372549019607843, 0.5372549019607843,
	0.5372549019607843} \definecolor{c}{rgb}{0.6352941176470588, 0.7843137254901961, 0.9254901960784314}
\definecolor{c6}{rgb}{1.0, 0.7372549019607844, 0.4745098039215686}
\definecolor{c7}{rgb}{0.8117647058823529, 0.8117647058823529,
	0.8117647058823529}

\newcommand{\logLogSlopeReverseTriangle}[5]
{
	
	\pgfplotsextra
	{
		\pgfkeysgetvalue{/pgfplots/xmin}{\xmin}
		\pgfkeysgetvalue{/pgfplots/xmax}{\xmax}
		\pgfkeysgetvalue{/pgfplots/ymin}{\ymin}
		\pgfkeysgetvalue{/pgfplots/ymax}{\ymax}
		
		\pgfmathsetmacro{\xArel}{#1}
		\pgfmathsetmacro{\yArel}{#3}
		\pgfmathsetmacro{\xBrel}{#1+#2}
		\pgfmathsetmacro{\yBrel}{\yArel}
		\pgfmathsetmacro{\xCrel}{\xArel}
		
		\pgfmathsetmacro{\lnxB}{\xmin*(1-(#1-#2))+\xmax*(#1-#2)} 
		\pgfmathsetmacro{\lnxA}{\xmin*(1-#1)+\xmax*#1} 
		\pgfmathsetmacro{\lnyA}{\ymin*(1-#3)+\ymax*#3} 
		\pgfmathsetmacro{\lnyC}{\lnyA+#4*(\lnxA-\lnxB)}
		\pgfmathsetmacro{\yCrel}{\lnyC-\ymin)/(\ymax-\ymin)}
		
		\coordinate (A) at (rel axis cs:\xArel,\yArel);
		\coordinate (B) at (rel axis cs:\xBrel,\yBrel);
		\coordinate (C) at (rel axis cs:\xCrel,\yCrel);
		
		\draw[#5]   (A)-- node[pos=0.5,anchor=north] {\scriptsize 1}
		(B)-- 
		(C)-- node[pos=0.5,anchor=east] {\scriptsize #4}
		cycle;
	}
}

\begin{document}



\title{\textbf{An \textit{a posteriori} error estimator for isogeometric analysis on trimmed geometries}}


\author{{A. Buffa$^{1,2}$, O. Chanon$^1$, R. V\'azquez$^{1,2}$}\\ \\
	\footnotesize{$^1$ MNS, Institute of Mathematics, \'Ecole Polytechnique F\'ed\'erale de Lausanne, Switzerland}\\
	\footnotesize{$^2$ Istituto di Matematica Applicata e Tecnologie Informatiche `E. Magenes' (CNR), Pavia, Italy}
}
 \date{August 8, 2022}

\maketitle
\vspace{-0.8cm}
\noindent\rule{\linewidth}{0.4pt}
\thispagestyle{fancy}

\begin{abstract}
{Trimming consists of cutting away parts of a geometric domain, without reconstructing a global parametrization (meshing). It is a widely used operation in computer aided design, which generates meshes that are unfitted with the described physical object. This paper develops an adaptive mesh refinement strategy on trimmed geometries in the context of hierarchical B-spline based isogeometric analysis. A residual \textit{a posteriori} estimator of the energy norm of the numerical approximation error is derived, in the context of \review{the} Poisson equation. \review{The estimator is proven to be reliable, independently \review{of}} the number of hierarchical levels and \review{of} the way the trimmed boundaries cut the underlying mesh. Numerical experiments are performed to validate the presented theory, \review{and to show that the estimator's effectivity index is independent \review{of} the size of the active part of the trimmed mesh elements}.}
\end{abstract}


\section{Introduction} \label{s:intro}
Since its introduction in 2005 in the pioneering work \cite{igabasis}, isogeometric analysis (IGA) has been a very successful area of research including mathematical discoveries, advances in computational mechanics, and efforts to tackle problems in the interface between different research communities. The ultimate goal of IGA is to provide a framework unifying geometric design and the analysis of partial differential equations (PDEs) for computational engineering problems in a single workflow. To do so, the idea is to use the same building blocks employed in \review{computer aided design} (CAD) for both phases, namely smooth B-splines, non-uniform rational B-splines (NURBS) and variants thereof \cite{igabook}. 
While many steps are still missing to achieve such an ambitious goal, a large amount of research has already been done on IGA: solid mathematical foundations have been built \cite{igahref,igaresume} and excellent properties in a wide range of mathematical and engineering applications have been shown. We refer to the special issue \cite{igaappli} for a review of some of the most important works on the subject. 

In other words, IGA stands for the class of methods using spline discretization techniques over a computational domain described as a collection of spline parametrizations. In standard CAD softwares, such spline parametrizations, called patches, are often trimmed in order to be able to deal with more complex geometric models. This process generates meshes that are unfitted with the described physical object, since no global parametrization (re-meshing) is constructed to fit the trimmed object. To obtain a unified end-to-end methodology between geometric design and analysis, it is thus crucial to properly address the challenges coming from the treatment of trimmed models in the analysis \cite{reviewtrimmedissues,schillinger2012isogeometric,trimmingbletzinger}. Several results have succeeded to overcome some of the issues arising from the analysis on trimmed geometries, such as the need for a reparametrization of the cut elements for integration purposes \cite{nagy2015numerical,liu2015weighted,antolinvreps,adaptiveintegration,momentsDuster,antolinweibuffa}, or the need of stabilization techniques to recover the well-posedness of the differential problem and the accuracy of its numerical solution \cite{ghostpenalty,stabcutiga,puppistabilization}. However, much work remains to be done.

In particular, it is widely known that the tensor product structure of B-splines hinders the capability of efficiently capturing localized features of the PDE solution in small areas of interest in the computational domain. Many research papers, reviewed in \cite{reviewadaptiveiga}, have thus tackled the important challenge of constructing locally refined splines, and using them within an adaptive paradigm. Research has now reached an advanced maturity on the subject. However, the use of locally refined splines on complex geometries defined by trimming operations is still at its first steps, and the aim of the present paper is to go further in that direction. More precisely, we are interested in the study of an adaptive framework using hierarchical B-splines (HB-splines) \cite{forsey1988hierarchical, kraft1997adaptive} and their variant called truncated hierarchical B-splines (THB-splines) \cite{thbgiannelli, giannelli2016thb}, in the context of trimmed geometries. In \cite{trimmedshells}, using \cite[Section~4.5]{hollig}, it is shown that in the presence of trimming, (T)HB-splines yield a linearly independent basis suitable for the analysis. 

The use of THB-splines in the context of \review{the} Poisson problem and linear elasticity is studied in \cite{trimmedconditioning,immersedigaprecond}, where an emphasis is put on stability issues and bad conditioning of the system matrix caused by trimming, but no study of error indicators is given.
To provide an approximation of the error, an implicit error estimator \review{was} introduced in \cite{trimmedshells} in the context of (T)HB-spline isogeometric analysis on trimmed surfaces, extending their previous work \cite{coradello}, on error estimation for linear fourth-order elliptic partial differential equations on non-trimmed geometries. The estimator relies on the solution of an additional residual-like system, but its reliability is not demonstrated. Our contribution differs from the aforementioned works as we introduce an explicit residual \textit{a posteriori} estimator of the energy norm of the numerical error of \review{the} Poisson problem in trimmed geometries, and a mathematical proof of its reliability is given. Driven by the proposed estimator, and thanks to the local refinement capability given by hierarchical B-splines, we develop a fully adaptive error-driven numerical framework for \review{the} Poisson problem in trimmed geometries of arbitrary dimension. In this paper, very general geometries are considered since the only hypothesis required on the trimming boundary is to be Lipschitz. Moreover, the \review{reliability of the estimator is proven to be} independent \review{of} the way the trimmed boundaries cut the underlying mesh, and thus in particular, it is independent \review{of} the size of the active part of the trimmed elements. \review{Numerical examples are given to also show the efficiency of the proposed estimator.}

When dealing with geometric domains defined by trimming, the main challenge lies in the fact that the generated mesh is unfitted with the described physical domain. Therefore, our work is strongly related to the immersed or unfitted mesh methods. 
\review{In particular, an implicit \textit{a posteriori} estimator of the energy error due to numerical approximation is introduced in \cite{sunshillinger} for any polynomial degree, but its efficiency is only demonstrated numerically. In \cite{he2019residual}, a reliable and efficient residual-based \textit{a posteriori} error estimator is studied for a partially penalized linear immersed finite element method applied to elliptic interface problems. In the context of the cut finite element method, an estimator is proposed in \cite{burman2019posteriori} for an elliptic model problem with Dirichlet boundary conditions ensured by a ghost penalty stabilization \cite{ghostpenalty}. In that paper, to avoid the dependence on the location of domain-mesh intersection, the efficiency for the term of ghost penalty is shown globally. In both \cite{he2019residual} and \cite{burman2019posteriori}, a linear finite element basis is considered, while in the very recent work \cite{chen2021adaptive}, a reliable and efficient $hp$-residual type error estimator is given in the case of high-order unfitted finite element for interface problems in the framework of the local discontinuous Galerkin method. The reliable estimator we introduce in the present article for any polynomial degree is very similar to the one in \cite{chen2021adaptive} for two-dimensional domains, but our contribution deals with more general two- and three-dimensional computational domains. In particular, while the notion of ``large element'' is central in \cite{chen2021adaptive}, we do not require any particular assumption on the way the trimming boundary intersects the underlying mesh. Thus in the proposed estimator, the scaling of the residuals with respect to the size of the trimmed mesh elements needs to be adapted.

This} paper is structured as follows. After introducing some notation in Section \ref{sec:notation}, and reviewing the fundamental concepts of isogeometric analysis with hierarchical B-splines in Section \ref{sec:iga}, Section \ref{sec:modelpb} introduces the considered elliptic model problem that is solved in a trimmed geometry. Then in Section \ref{sec:estimator}, we introduce the \textit{a posteriori} estimator of the numerical error, in \review{the} energy norm, and its reliability is proven. In Section \ref{sec:adapt}, we briefly introduce the adaptive mesh refinement strategy in the context of trimmed geometries, before presenting several numerical experiments in Section \ref{sec:numexp}. Finally, Section \ref{sec:ccl} draws some conclusions, and mathematical results used throughout the paper are stated and proven in the Appendix. 

\section{Notation} \label{sec:notation}
We start by introducing the notation that will be used throughout the paper. Let $n=2$ or $n=3$, let $\omega$ be an open $k$-dimensional manifold in $\mathbb R^n$, $k\leq n$, and let $\varphi \subset \partial \omega$. We denote by $|\omega|$, $\overline{\omega}$ and int$(\omega)$, respectively, the measure of $\omega$, its closure and its interior. Moreover, we write diam$(\omega)$ \review{for} the diameter of $\omega$, that is, diam$(\omega):= \max_{\xi,\eta\in\omega} \rho(\xi,\eta)$, where $\rho(\xi,\eta)$ is the infimum of lengths of continuous piecewise $C^1$-paths between $\xi$ and $\eta$. 

Furthermore, for $1\leq p \leq \infty$, let $\|\cdot\|_{L^p(\omega)}$ be the norm in $L^p(\omega)$, and let $H^s(\omega)$ denote the Sobolev space of order $s\in \mathbb R$ whose classical norm and semi-norm are written as $\|\cdot\|_{s,\omega}$ and $|\cdot|_{s,\omega}$, respectively. We also write $L^2(\omega):= H^{0}(\omega)$, so that the norm in $L^2(\omega)$ will be written as $\|\cdot\|_{0,\omega}$. To deal with boundary conditions, for $z\in H^\frac{1}{2}(\varphi)$, we denote
\review{\begin{equation} \label{eq:H1spacewithtrace}
	H^1_{z,\varphi}(\omega) := \left\{ y\in H^1(\omega) : \text{tr}_\varphi(y) = z \right\},
	\end{equation}}
where tr$_\varphi(y)$ denotes the trace of $y$ on $\varphi \subset \partial \omega$. \\

In the remaining part of this article, the symbol $\lesssim$ indicates an inequality hiding a constant which does not depend on the mesh size $h$, on the size of the active part of the trimmed elements, nor on the number of hierarchical levels $L$. However, those inequalities may depend on the shape of the mesh elements. Moreover, we will write $A\simeq B$ whenever $A\lesssim B$ and $B\lesssim A$.

\section{Isogeometric analysis} \label{sec:iga}
\input{iga}

 \section{Trimming model problem} \label{sec:modelpb}
 \input{modelproblem}

\section{An \textit{a posteriori} error estimator on trimmed geometries} \label{sec:estimator}
\input{estimator}

\section{An $h$-refinement adaptive strategy on trimmed geometries}\label{sec:adapt}
\input{adaptivity}

\section{Numerical experiments}\label{sec:numexp}
 \input{numtests}

\section{Conclusions} \label{sec:ccl}
We have introduced a novel residual \textit{a posteriori} estimator of the energy error coming from the numerical approximation of \review{the} Poisson problem on a trimmed geometry of arbitrary dimension, using hierarchical B-spline based isogeometric analysis. We have demonstrated its reliability \review{independently of the trimming boundary}, and we have tested it on an extensive set of numerical experiments: in all of them, we have observed that the proposed estimator acts as an excellent approximation of the true error. \review{We have} numerically verified that the effectivity index of the estimator is also independent \review{of} the type of cut, and in particular, it is independent \review{of} the measure of the active part of the trimmed elements. 

Moreover, the proposed estimator being naturally decomposed into local element contributions, we have introduced an adaptive mesh refinement strategy on trimmed geometries, driven by this estimator. This method strongly relies on the adaptive strategies introduced for non-trimmed geometries, that also require an admissibility assumption of the underlying mesh. We have performed different numerical experiments which exhibit both smooth and singular solutions, where optimal asymptotic rates of convergence are recovered. Compared to uniform refinement, the adaptive strategy exhibits a substantial increase in accuracy with respect to the number of degrees of freedom, as it has already been observed in the literature for non-trimmed geometries. 

\renewcommand{\thesection}{A}
\section{Appendix}\label{sec:appendix}
\input{appendix}

\section*{Acknowledgment}
The authors gratefully acknowledge the support of the European Research Council, via the ERC AdG project CHANGE n.694515. 
R. V\'azquez also thanks the support of the Swiss National Science Foundation via the project HOGAEMS n.200021\_188589.

\addcontentsline{toc}{chapter}{Bibliography}
\bibliography{estimator_trimming}
\bibliographystyle{ieeetr}

\end{document}

%% file: iga.tex
In this section, we briefly review the notations and basic concepts related to B-splines and isogeometric analysis (IGA), following \cite{igaresume}. For a detailed review of the method and its applications, the reader is referred to \cite{igabook}. 

\subsection{Standard B-splines}\label{ss:sbs}
Let $p$, $N_\text{dof}$ $\in \mathbb{N}\setminus\{0\}$, where $p$ will denote the degree of the B-spline basis functions and $N_\text{dof}$ their number, also called number of {degrees of freedom}. Then, let $\Xi = \{0=\xi_1,\xi_2, \ldots, \xi_{N_\text{dof}+p+1}=1\}$  be a non-decreasing sequence of real values in the parametric space $(0,1)$, referred to as {knot vector}. 
\review{We call multiplicity the number of times a knot $\xi\in\Xi$ is repeated in the knot vector. As it is common practice in standard CAD, we assume in this work that knot vectors are always open, meaning that the first and last knots have multiplicity $p+1$. This leads to an interpolatory B-spline basis at both ends of the parametric interval. The univariate B-spline basis $\hat{\mathcal{B}}_p = \left\{\hat B_{i,p}:(0,1)\to\mathbb{R}\right\}_{i=1}^{N_\text{dof}}$ corresponding to $\Xi$ is then} defined recursively using \review{the} Cox-de Boor formula \cite{coxdeboor}. The obtained B-spline basis is $C^{p-k}$-continuous at every knot, where $k\in \mathbb{N}\setminus\{0\}$ is the multiplicity of the considered knot, and $C^\infty$-continuous elsewhere. 

Multivariate B-spline basis functions are defined in a straight-forward manner from tensor-products of univariate B-splines. That is, let $\mathbf p = \left(p_1, \ldots, p_{n}\right)$ be the vector of polynomial degrees in each of the $n\in \{1,2,3\}$ dimensions, and for all $j=1,\ldots,n$, let $\Xi^j := \{0=\xi_1^j,\xi_2^j,\ldots, \xi_{{N_\text{dof}}_j+p_j+1}^j=1\}$ be the knot vector corresponding to the parametric direction $j$. 
Then if $\hat B^j_{i,p_j}$ is the $i$-th basis function in the direction $j\in\{1,\ldots,n\}$, and if $\mathbf i = \left(i_1, \ldots, i_{{n}}\right)\in \mathbf I$ is a multi-index denoting a position in the tensor-product structure (using the previous notation, $\mathbf I = \{1,\ldots,N_\text{dof}\}$ if $n=1$), then the multivariate B-spline basis of degree $\mathbf p$ is defined by $$\hat{\mathcal{B}}_\mathbf p := \left\{\hat B_{\mathbf i,\mathbf p} : \mathbf i\in \mathbf I \text{ and } \hat B_{\mathbf i,\mathbf p} := \review{\prod_{j=1}^{n}} \hat B_{i_j, p_j}^j : (0,1)^{n}\to\mathbb{R} \right\}.$$

B-spline geometries (curves, surfaces, volumes) are then constructed as a linear combination of the basis functions $\hat B_{\mathbf i,\mathbf p}$ and some {control points} $\{\mathbf{P}_\mathbf i\}_{\mathbf i\in\mathbf I}\subset \mathbb{R}^n$ of the physical space. That is, if we let $\hat{\Omega}_0 := (0,1)^{n}$ be the parametric space, a B-spline geometry is defined as the image of a mapping $\mathbf F:\hat{\Omega}_0\to\mathbb R^n$ defined by $\mathbf{F}(\boldsymbol\xi) = \sum_{\mathbf i\in\mathbf I} \hat B_{\mathbf i,\mathbf p}(\boldsymbol\xi)\mathbf{P}_\mathbf i$ for all $\boldsymbol\xi \in \hat{\Omega}_0$.

For $j = 1,\ldots,n$, let $\tilde\Xi^j := \left\{\zeta_1^j, \ldots, \zeta^j_{M_j}\right\}$ be the set of non-repeated knots of $\Xi^j$, where $M_j\in \mathbb{N}\setminus\{0,1\}$ denote their number. Then the values of $\tilde\Xi^j$, also called {breakpoints}, form a Cartesian grid in the parametric domain $\hat{\Omega}_0$ called {parametric B\'ezier mesh} that we denote
$$\hat{\mathcal Q} := \left\{\hat K_{\mathbf m} := \displaystyle\bigtimes_{j=1}^{n} \left(\zeta_{m_j}^j, \zeta_{m_j+1}^j\right) : \mathbf m=\review{(}m_1,\ldots,m_{n}\review{)}, 1\leq m_j\leq M_j-1 \text{ for } j=1,\ldots,n\right\}.$$ 
Each $\hat K_\mathbf m:= \displaystyle\bigtimes_{j=1}^{n} \left(\zeta_{m_j}^j, \zeta_{m_j+1}^j\right)$ is called {parametric element} or cell, and we define its support extension to be
\begin{equation}\label{eq:suppext}
S_\text{ext}\left(\hat K_\mathbf m\right) := \bigtimes_{j=1}^{n}S_\text{ext}\left(\zeta_{m_j}^j, \zeta_{m_j+1}^j\right),
\end{equation}
where if $i$ is the index, $p_j+1\leq i \leq {N_\text{dof}}_j$, such that we can uniquely rewrite $\left(\zeta_{m_j}^j, \zeta_{m_j+1}^j\right) = \left(\xi_i^j,\xi^j_{i+1}\right)$, then
$$S_\text{ext}\left(\zeta_{m_j}^j, \zeta_{m_j+1}^j\right) := \review{\left(\xi_{i-p_j}^j, \xi_{i+p_j+1}^j\right)}.$$

Remark that all concepts introduced in this section can be readily transferred to non-uniform rational B-splines; the interested reader is referred to \cite{igabook,nurbsbook}.

\subsection{Hierarchical B-splines}\label{ss:hbs}
In this sub-section, we briefly introduce the concept of hierarchical B-splines (HB-splines), following \cite{thb2}. For more details, the reader is also referred to \cite{vuong}.

Let $\hat\Omega_0=(0,1)^{n}$ as before, let $L\in\mathbb N$, and let $\hat{\mathcal{B}}^0, \hat{\mathcal{B}}^1, \ldots, \hat{\mathcal{B}}^L$ be B-spline \review{bases} on $\hat \Omega_0$ such that 
\begin{equation} \label{eq:nestedBsp}
\spn{\hat{\mathcal{B}}^0}\subset \spn{\hat{\mathcal{B}}^1}\subset \ldots \subset \spn{\hat{\mathcal{B}}^L}.
\end{equation}
For all $\ell=0,\ldots,L$, let us denote $\hat{\mathcal Q}^\ell$ the mesh associated with $\hat{\mathcal B}^\ell$, where $\hat K\in \hat{\mathcal Q}^\ell$ represents a cell of level $\ell$. 
Then, let $\hat{\underline{\Omega}}_0^L:= \{\hat\Omega_0^0,\hat\Omega_0^1,\ldots,\hat\Omega_0^L\}$ be a hierarchy of sub-domains of $\hat\Omega_0$ of depth $L$, that is, such that
\begin{itemize}
	\item $\hat \Omega_0 =: \hat\Omega_0^0 \supseteq \hat\Omega_0^1 \supseteq \ldots \supseteq \hat\Omega_0^L := \emptyset$,
	\item each sub-domain $\hat \Omega_0^\ell$ is \review{a union of cells of level $\ell-1$, i.e. 
		$$\hat \Omega_0^\ell := \mathrm{int}\left( \bigcup_{\hat K \in \hat{\mathcal Q}^{\ell-1}_*} \overline{\hat K}\right), \quad \hat{\mathcal Q}^{\ell-1}_* \subset \hat{\mathcal Q}^{\ell-1}, \quad \forall \ell=1,\ldots,L.$$}
\end{itemize}
The HB-spline basis $\hat{\mathcal{H}}=\hat{\mathcal{H}}\left(\underline{\hat\Omega}_0^L\right)$ is then defined recursively in the following way: 
\begin{equation*}
\begin{cases}
\hat{\mathcal{H}}^0 := \hat{\mathcal{B}}^0; \\
\hat{\mathcal{H}}^{\ell+1} := \hat{\mathcal{H}}^{\ell+1}_A \cup \hat{\mathcal{H}}^{\ell+1}_B, \hspace{2mm} \ell = 0,\ldots,L-2, \\
\quad \text{with } \hat{\mathcal{H}}^{\ell+1}_A := \left\{\hat B\in \hat{\mathcal{H}}^\ell : \hspace{1mm} \supp{\hat B} \not\subseteq \hat\Omega_0^{\ell+1}\right\} \text{ and } \, \hat{\mathcal{H}}^{\ell+1}_B := \left\{\hat B\in\hat{\mathcal{B}}^{\ell+1} : \hspace{1mm} \supp{\hat B} \subseteq \hat\Omega_0^{\ell+1} \right\};\\
\hat{\mathcal{H}} := \hat{\mathcal{H}}^{L-1},
\end{cases}
\end{equation*}
\review{where the supports are considered to be open.}

Namely, HB-spline basis functions are the B-splines of each level $\ell$ whose support covers just elements of level $\ell^\star \geq \ell$ and at least one element of level $\ell$. Those basis functions, i.e. in $\hat{\mathcal H}\cap\hat{\mathcal B}^\ell$, are referred to as active basis functions of level $\ell$. This allows us to build a basis that is locally refined and we can therefore overcome the limitations intrinsic to the tensor-product nature of B-splines. 

In the same manner as B-spline curves, surfaces or volumes, hierarchical geometries are built as a linear combination of the HB-spline basis functions and control points of the physical space. 
Moreover, the {parametric hierarchical mesh} associated to the domain hierarchy $\hat{\underline{\Omega}}_0^L$ is defined as 
$$\hat{\mathcal Q} := \bigcup_{\ell=0}^L \hat{\mathcal Q}^\ell_A, \quad \text{ with } \quad \hat{\mathcal Q}^\ell_A := \left\{ \hat K\in\hat{\mathcal Q}^\ell : \hat K\subseteq \hat{\Omega}_0^\ell, \,\hat K\not\subseteq \hat\Omega_0^{\ell+1}\right\}.$$
We refer to $\hat K\in \hat{\mathcal Q}$ as an active cell, to $\hat K\in \hat{\mathcal Q}^\ell_A$ as an active cell of level $\ell$, \review{and we write lev$\Big(\hat K\Big)=\ell$}. Furthermore, if $\Omega_0$ is a HB-spline geometry determined by (the refinement of) a mapping $\mathbf F:\hat \Omega_0 \to \Omega_0$, we define the corresponding physical hierarchical mesh as
\begin{equation} \label{eq:mesh}
\mathcal Q(\Omega_0) := \left\{ K := \mathbf F\left(\hat K\right) : \hat K\in \hat{\mathcal Q} \right\}.
\end{equation}

To prevent the existence of singularities in the parametrization, it is standard to make the following assumption, \review{which} we assume to be satisfied in the \review{remainder} of this work:
\begin{enumerate}
	\item[(A0)] The mapping $\mathbf F : \hat{\Omega}_0 \to \Omega_0$ is bi-Lipschitz, $\mathbf F\vert_{\overline{\hat K}} \in C^\infty\left(\overline{\hat K}\right)$ for every $\hat K \in \hat{\mathcal Q}$, and $\mathbf F^{-1}\vert_{\overline{K}} \in C^\infty\left(\overline{K}\right)$ for every $K \in \mathcal Q(\Omega_0)$.\\
\end{enumerate}

Finally, we introduce the notion of $\mathcal T$-admissibility for a hierarchical mesh $\mathcal Q(\Omega_0)$, following \cite{reviewadaptiveiga}, and the element neighborhood in the hierarchical context, following \cite{bracco}. To do so, let us first extend the notion of support extension (\ref{eq:suppext}) by defining the multilevel support extension of an element $\hat K\in\hat{\mathcal Q}^\ell$ with respect to level $k$, with $0\leq k\leq \ell\leq L$, as follows:
$$S_\text{ext}\left(\hat K, k\right) := S_\text{ext}\left(\hat{K}'\right), \quad \text{with } \hat K'\in \hat{\mathcal Q}^k \text{ and } \hat{K} \subseteq \hat{K}'.
$$
Then, define the auxiliary domains $\hat \omega^0 := \hat \Omega_0^0 = \hat \Omega_0$, and for $\ell = 1,\ldots,L$, 
$$\hat \omega^\ell := \bigcup \left\{ \overline{\hat K} : \hat K\in \hat{\mathcal Q}^\ell, \,S_\text{ext}\left(\hat K, \ell\right)\subseteq \hat{\Omega}_0^\ell \right\}.$$
By the nested nature of the B-spline spaces in (\ref{eq:nestedBsp}), one can express $\hat{B}\in \hat{\mathcal B}^{\ell-1}$ in terms of B-spline basis functions of level $\ell$ as 
$$ \hat{B} = \sum_{\hat{B}_i^{\ell}\in \hat{\mathcal B}^\ell} c_i \hat{B}_i^\ell.$$
If we define the truncation of $\hat B$ with respect to $\ell$ as 
$$ \text{trunc}^\ell\hat{B} := \sum_{\hat{B}_i^{\ell}\in \hat{\mathcal B}^\ell\setminus \hat{\mathcal H}_B^\ell} c_i \hat{B}_i^\ell,$$
then the domains $\hat \omega^\ell $ represent the regions of $\hat \Omega_0^\ell$ where all the active basis functions of level $\ell-1$ truncated with respect to level $\ell$ \review{are} equal to zero. 

We are now able to introduce the following definitions:
\begin{definition} \label{def:admissibility}
	\begin{itemize}
		\item A parametric mesh $\hat{\mathcal Q}$ is $\mathcal T$-admissible of class $\mu\in\{\review{2},\ldots,L-1\}$ if 
		$$\forall\ell = \mu, \mu+1, \ldots, L-1, \quad \hat \Omega_0^\ell \subseteq \hat \omega^{\ell-\mu+1}.$$
		From \cite[Proposition~9]{buffagiannelli1}, this \review{implies} that the truncated hierarchical B-spline (THB-spline) basis functions in 
		\begin{equation} \label{eq:thbsplinebasis}
		\hat{\mathcal T} := \left\{ \text{trunc}^{\review{L}}\left( \ldots \left(\text{trunc}^{\ell+1}\hat{B}\right)\cdots\right) : \hat{B}\in \hat{\mathcal{B}}^{\ell}\cap\hat{\mathcal H}, \,\ell = 0,\ldots,L-1 \right\}
		\end{equation}
		which take nonzero values over any element $\hat K\in \hat{\mathcal Q}$ belong to at most $\mu$ successive levels.
		\item A physical mesh $\mathcal Q(\Omega_0)$ defined by (\ref{eq:mesh}) is $\mathcal T$-admissible of class $\mu\in\{\review{2},\ldots,L-1\}$ if the underlying parametric mesh $\hat{\mathcal Q}$ is $\mathcal T$-admissible of class $\mu$. 
		\item For all $\hat K\in \hat{\mathcal Q}\cap\hat{\mathcal Q}_A^\ell$, the neighborhood of $\hat K$ with respect to $\mu\in\{\review{2},\ldots,L\review{-1}\}$ is the set
		$$\mathcal N\left(\hat{\mathcal Q}, \hat K, \mu\right) := \left\{ \hat K'\in \hat{\mathcal Q}_A^{\ell-\mu+1} : \hat K'\review{\cap} S_\text{ext}\left(\hat K,\ell-\mu+2\right)\review{\neq \emptyset} \right\}$$
		if $\ell-\mu+1\geq 0$, and $\mathcal N\left(\hat{\mathcal Q}, \hat K, \mu\right) := \emptyset$ otherwise, that is, if $\ell-\mu+1<0$. 
		\item The neighborhood of a physical element $K\in \mathcal Q(\Omega_0)$ is the set of the push-forward elements of the neighborhood of the pull-back $\mathbf F^{-1}(K)$ of $K$, that is, 
		$$\mathcal N\big(\mathcal Q(\Omega_0), K,\mu\big) := \left\{ \mathbf F\left(\hat K'\right) : \hat K'\in \mathcal N\left(\hat{\mathcal Q}, \mathbf F^{-1}(K), \mu\right) \right\}.$$
	\end{itemize}
\end{definition}

\subsection{Isogeometric paradigm}
To solve a partial differential equation (PDE), the idea behind isogeometric analysis is to use the same basis functions both to describe the computational domain and to define the finite dimensional space on which the Galerkin solution of the PDE is sought. In this paper, we concentrate on HB-spline based IGA. That is, if we let $\mathbf{F}$ be (a refinement of) the parametrization map of the computational domain $\Omega_0$ generated by the HB-spline basis $\hat{\mathcal{H}}$, then the numerical solution is sought in the finite dimensional space spanned by
$$
\mathcal{H}(\Omega_0) = \left\{B:=\hat B\circ \mathbf{F}^{-1} :  \hat B\in\hat{\mathcal{H}}\right\}.
$$

%% file: modelproblem.tex
Let $n\in\left\{2,3\right\}$ and let $\Omega_0\subset \mathbb{R}^n$ be a hierarchical geometry as defined in Section \ref{sec:iga} defined from a hierarchy $\underline{\hat{\Omega}}_0^L$ with $L\in\mathbb N$. That is, $\Omega_0$ is the image of a bijective bi-Lipschitz map $\mathbf F : \hat{\Omega}_0 \rightarrow \Omega_0$ generated by an HB-spline basis $\hat{\mathcal H}$, where $\hat\Omega_0 = (0,1)^{n}$ is the parametric domain. Moreover, let $\left\{D_i\right\}_{i=1}^s$ be a set of bounded open domains in $\mathbb R^n$ that are trimmed from $\Omega_0$ in order to obtain the computational domain $\Omega$, i.e.,
$$\Omega := \Omega_0 \setminus \overline{D}, \quad \text{ with } \quad D:= \text{int}\left(\bigcup_{i=1}^s \overline{D_i}\right),$$
as illustrated in Figure \ref{fig:domains}.
Assume that $\Omega$ is an open Lipschitz domain, let $\mathbf n$ be the unitary outward normal of $\partial \Omega$, and let $\Gamma_N,\Gamma_D \subset \partial \Omega$ be open such that $\Gamma_D \cap \Gamma_N = \emptyset$, $\overline{\Gamma_D} \review{\cup} \overline{\Gamma_N} = \partial \Omega$ and $\Gamma_D\neq \emptyset$. Note that since $\Omega$ is Lipschitz, \review{the} trimming boundary $\gamma := \partial \Omega \setminus \overline{\partial \Omega_0}$ is also Lipschitz.

\begin{figure}
	\centering
	\begin{subfigure}{0.4\textwidth}
		\begin{center}
			\begin{tikzpicture}[scale=2.5]
			\draw[thick] (2,0) -- (3,0) ;
			\draw[c2,thick] (2,0) -- (2,1) ;
			\draw[c2,thick] (3,0) -- (3,1) ;
			\draw[c2,thick] (2,1) -- (3,1) ;
			\draw (2.5,0.5) node{$\Omega_0$} ;
			\draw (2.5,0) node[below]{$\Gamma_D$} ; 
			\draw[c2] (2.9,1) node[above]{$\Gamma_N^0$} ; 
			\end{tikzpicture}
			\caption{Non-trimmed domain $\Omega_0$.}
			\label{fig:Omega0}
		\end{center}
	\end{subfigure}
	~ 
	\begin{subfigure}{0.4\textwidth}
		\begin{center}
			\begin{tikzpicture}[scale=2.5]
			\draw[thick] (2,0) -- (3,0) ;
			\draw[c2,thick] (2,0) -- (2,1) ;
			\draw[c2,thick] (3,0) -- (3,1) ;
			\draw[c2,thick] (2,1) -- (3,1) ;
			\draw[c1,thick] (2.2, 1) arc (-180:0:.3cm);
			\fill[pattern=north west lines, pattern color=c1, opacity=0.5] (2,0) -- (2,1) -- (2.2,1) arc (-180:0:.3cm) -- (3,1) -- (3,0) -- (2,0);
			\draw[c1] (2.5,0.25) node{$\Omega$} ;
			\fill[pattern=north east lines, pattern color=c2, opacity=0.5] (2.2, 1) arc (-180:0:.3cm) -- (2.2,1);
			\draw[c2] (2.5,0.85) node{$D$} ;
			\draw (2.5,0) node[below]{$\Gamma_D$} ; 
			\draw[c2] (2.9,1) node[above]{$\Gamma_N^0$} ; 
			\draw[c1] (2.25,0.7) node{$\gamma$} ;
			\end{tikzpicture}
			\caption{Trimmed domain $\Omega$ and $D:=\Omega_0\setminus\overline\Omega$.}
			\label{fig:OmegaD}
		\end{center}
	\end{subfigure}
	\caption{Illustration of the notation used on a trimmed geometry.}
	\label{fig:domains}
\end{figure}
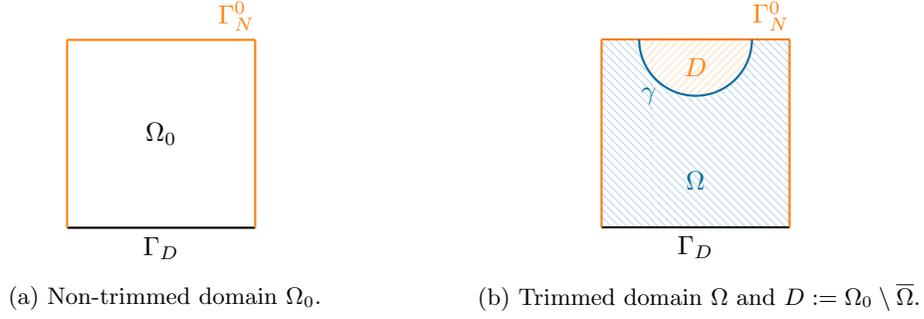

\subsection{Continuous formulation}
Let \review{$g_D\in H^{\frac{1}{2}}(\Gamma_D)$, $g_N\in L^2(\Gamma_N)$} and $f\in L^2\left(\Omega\right)$, and consider the following Poisson equation in $\Omega$: find $u\in H^1(\Omega)$, the weak solution of 
\begin{align} \label{eq:originalpb}
\begin{cases}
-\Delta u = f &\text{ in } \Omega \\
u = g_D &\text{ on } \Gamma_D \\
\displaystyle\frac{\partial u}{\partial \mathbf{n}} = g_N &\text{ on } \Gamma_N,\vspace{0.1cm}
\end{cases}
\end{align}
that is, \review{following the notation in~(\ref{eq:H1spacewithtrace}),} $u\in H^1_{g_D,\Gamma_D}(\Omega)$ satisfies for all $v\in H^1_{0,\Gamma_D}(\Omega)$, 
\begin{equation} \label{eq:weakpb}
\int_\Omega \nabla u \cdot \nabla v \,\mathrm dx = \int_\Omega fv \,\mathrm dx + \int_{\Gamma_N} \review{g_N} v \,\mathrm ds.
\end{equation}

To simplify the subsequent analysis, we suppose that $\overline{\Gamma_D} \cap \overline \gamma = \emptyset$, that is, we suppose that the Dirichlet boundary is not part of the trimming boundary $\gamma$. If it were not the case, then in the discrete setting, we would need to weakly impose the Dirichlet boundary conditions on $\overline{\Gamma_D} \cap \overline \gamma$, and make use of stabilization techniques \cite{ghostpenalty,stabcutiga,puppistabilization}. \review{According} to Lax-Milgram theorem, problem (\ref{eq:weakpb}) admits a unique solution $u\in H^1_{g_D,\Gamma_D}(\Omega)$. 

\subsection{HB-spline based IGA formulation} \label{sec:HBIGAformulation}
\review{Let $\mathcal Q := \mathcal Q\left(\Omega_0\right)$ be (a refinement of) the hierarchical physical mesh on $\Omega_0$ as defined in~(\ref{eq:mesh}), and let $\mathcal{F}$ be the set of all faces of $\mathcal{Q}$. Note that here and in the sequel, edges are called faces when $n=2$, and assume that the Dirichlet boundary $\Gamma_D$ is the union of full element faces. Moreover, let $h_K:=\text{diam}(K)$ for all $K\in \mathcal Q$, let $h_F := \text{diam}(F)$ for all $F\in\mathcal F$, and let $h:=\displaystyle\max_{K\in \mathcal Q} h_K$.}

Furthermore, let $\mathcal Q_\Omega := \mathcal Q_\trim \cup \mathcal Q_\untr$ be the active mesh intersecting $\Omega$, where
$$\mathcal Q_\trim := \left\{K\in \mathcal Q : K\cap \Omega \neq K, |K\cap\Omega| > 0 \right\} $$
is the set of cut (trimmed) elements, and
$$\mathcal Q_\untr := \left\{K\in \mathcal Q : K\subset \Omega\right\}$$
is the set composed of the other active (non-trimmed) elements in $\Omega$. 
Similarly, \review{let $\mathcal F_N := \mathcal F_\trim \cup \mathcal F_\untr$}, where
$$\mathcal F_\trim := \left\{ F\in \mathcal F : F\cap\Gamma_N\neq F \right\} \qquad \text{and} \qquad \mathcal F_\untr := \left\{ F \in \mathcal F : F\subset \Gamma_N \right\}.$$
\review{Note that $\mathcal{F}_N$ is therefore the set of faces~$F$ of $\mathcal{F}$ such that $|F\cap \Gamma_N|>0$. }
Recall that the trimming boundary is defined as $\gamma := \partial \Omega \setminus \overline{\partial \Omega_0} \subset \Gamma_N$, and for all $K\in\mathcal Q_\Omega$, let $$\gamma_K := \gamma\cap \text{int}(K).$$
Note that for all $K\in\mathcal Q_\untr$, $\gamma_K = \emptyset$. Also note that for all internal faces $F\in \mathcal F_N$ such that $\left(F\cap\Gamma_N\right) \subset \gamma$ and $\left(F\cap\Gamma_N\right)\subset \partial \left(K\cap\Omega\right)$ for some $K\in\review{\mathcal Q_\Omega}$, then $\gamma_K \cap \left(F\cap\Gamma_N\right) = \emptyset$, that is, this internal face is not included in $\gamma_K$. Or in other words, 
$$\review{\Gamma_N = \mathrm{int}\left[ \left(\bigcup_{F\in\mathcal F_N}\overline{F\cap\Gamma_N}\right) \cup \left(\bigcup_{K\in \mathcal Q_\trim} \overline{\gamma_K}\right)\right],}$$
where the intersection of any two elements of the union is empty, and each contribution of the union only appears once. \\

Now, let us make the following assumptions on $\mathcal Q$: 
\begin{enumerate}
	\item[(A1)] $\mathcal Q$ is shape regular, that is, for all $K\in \mathcal Q$, $\displaystyle\frac{h_K}{\rho_K} \lesssim 1$, where $\rho_K$ denotes the radius of the largest inscribed ball in $K$. As a consequence, $h_K\simeq h_F$ for all $K\in \mathcal Q$ and all \review{$F\in \mathcal F$ with $F\subset \partial K$}. 
	\item[(A2)] $\mathcal Q$ is $\mathcal T$-admissible of some fixed class $\mu\in \mathbb{N}$, \review{$\mu\geq 2$}, following Definition \ref{def:admissibility}.
\end{enumerate}

\begin{remark}
	Note that no further assumption on the trimming boundary $\gamma$ is required, other \review{than} $\gamma$ to be Lipschitz. In particular, and as already pointed out in \cite{guzman}, the literature on unfitted finite elements commonly imposes an additional restriction on how $\gamma$ intersects the mesh $\mathcal Q$. But we do not need here this assumption which, in $2$D, requires that $\gamma$ does not intersect an edge of the mesh $\mathcal Q$ more than once, and which is analogous in $3$D, see, e.g., \cite{hansbo2}. So for example, in this paper, $\gamma_K$ could be disconnected.
\end{remark}

Following the isogeometric paradigm introduced in Section~\ref{sec:iga} and generalizing it to trimmed geometries, let us define
\begin{align}
\mathcal{H} &:= \left\{B:=\hat B\circ \mathbf{F}^{-1} : \hat B\in\hat{\mathcal{H}}\right\},\nonumber \\
\mathcal{H}_\Omega &:= \left\{B\in \mathcal{H} : \supp{B}\cap \overline\Omega \neq \emptyset\right\}, \label{eq:HOmegabasis}
\end{align}
and let us also define the following approximation spaces: 
\begin{align*}
&V_h(\Omega_0) := \spn{\mathcal H}, \quad
V_h(\Omega) := \spn{B\vert_\Omega: B\in\mathcal H_\Omega}, \\
&\review{V_h^0(\Omega_0):= \spn{B\in\mathcal H : B\vert_{\Gamma_D} = 0} = \left\{ v_h \in V_h(\Omega_0) : v_h\vert_{\Gamma_D} = 0 \right\},} \\
&V_h^0(\Omega) := \spn{B\vert_\Omega: B\in\mathcal H_\Omega, B\vert_{\Gamma_D} = 0} = \left\{ v_h \in V_h(\Omega) : v_h\vert_{\Gamma_D} = 0 \right\}.
\end{align*} 
\review{Note that a proof of the linear independence of the trimmed basis $\mathcal{H}_\Omega$ in $\Omega$ can be found in \cite[Section~4.5]{hollig}.} 

\review{In the following, we assume that $g_D$ is the trace of a discrete function in $V_h(\Omega)$, that we still write $g_D$ by abuse of notation.} Then, the Galerkin method with finite basis $\mathcal H_\Omega$ is used to discretize weak problem (\ref{eq:weakpb}) and reads as follows: find $u_h\in g_D+V_h^{0}(\Omega)$ such that for all $v_h\in V_h^0(\Omega)$, 
\begin{equation} \label{eq:discrpb}
\int_\Omega \nabla u_h \cdot \nabla v_h \,\mathrm dx = \int_\Omega fv_h \,\mathrm dx + \int_{\Gamma_N} \review{g_N} v_h \,\mathrm ds.
\end{equation}
In the following, we are interested in the \textit{a posteriori} estimation of the energy error $\|\nabla (u-u_h)\|_{0,\Omega}$ on the trimmed geometry $\Omega$. 

\review{\begin{remark}
		The analysis is performed on this simple model Poisson problem for simplicity, but it can be readily extended to the estimation of the energy error for a general steady elliptic diffusion-advection-reaction problem or to a linear elasticity problem, as soon as they verify the assumptions of the Lax-Milgram theorem. 
\end{remark}}

%% file: estimator.tex
In this section, we derive an \textit{a posteriori} estimator of the energy error in the trimmed geometry $\Omega$ between the exact solution $u$ and the discrete solution $u_h$, and we prove its reliability. 

In the subsequent analysis, we assume for simplicity that the discrete functions are $C^1$-continuous. This assumption is not needed, but it allows us to simplify the analysis while \review{underlining} the specificity of IGA with higher order B-splines. The general case \review{of $C^0$-continuous basis functions} could be treated in a similar way \review{through the introduction of appropriate jump terms}, following the classical theory of \review{the} adaptive finite element method. \review{However,} the addition of face jumps does not add any relevant additional insight to the analysis.

So more precisely, let 
\begin{align} \label{eq:deltas}
\delta_K := \begin{cases}
h_K & \text{ if } K\in\mathcal Q_\untr \\
c_{K\cap\Omega} |K\cap\Omega|^\frac{1}{n} & \text{ if } K\in\mathcal Q_\trim, 
\end{cases}
\quad \text{ and } \quad
\delta_F := \begin{cases}
h_F^\frac{1}{2} & \text{ if } \review{F}\in\mathcal F_\untr \\
c_{F\cap\Gamma_N} |F\cap\Gamma_N|^\frac{1}{2(n-1)} & \text{ if } F\in\mathcal F_\trim,
\end{cases}
\end{align}
where, if we define $\eta\in\mathbb{R}$, \review{$\eta>0$}, as the unique solution of $\eta = -\log(\eta)$, \review{then}
\begin{align} \label{eq:constant}
c_{S} := \begin{cases}
\max\big(\hspace{-0.08cm}-\log\left(|S|\right), \eta \big)^\frac{1}{2} & \text{ if } n = 2 \\
1 & \text{ if } n = 3,
\end{cases}
\end{align}
for $S=K\cap\Omega$ for some $K\in\mathcal Q_\trim$ or $S=F\cap\Gamma_N$ for some $F\in\mathcal F_\trim$. Then in Theorem \ref{thm:upperbound}, we \review{will} show that 
\begin{align} \label{eq:estimator}
\mathscr E(u_h) &:= \left[ \sum_{K\in\mathcal Q_\Omega} \delta_K^2 \left\| f+\Delta u_h\right\|_{0,K\cap\Omega}^2 + \sum_{F\in\mathcal F_N} \delta_F^2 \left\|\review{g_N}-\frac{\partial u_h}{\partial \mathbf n}\right\|^2_{0,F\cap\Gamma_N} + \sum_{K\in \mathcal Q_\trim} h_K \left\|\review{g_N}-\frac{\partial u_h}{\partial \mathbf n}\right\|^2_{0,\gamma_K} \right]^\frac{1}{2}
\end{align}
is an upper bound of the energy error. \\

Before stating and proving this theorem, let us recall the following well-known scaled trace inequality, see, e.g., \cite[Section~1.4.3]{dipietroern}:
\begin{equation} \label{eq:trace}
\|v\|^2_{0,\partial K} \lesssim h_K^{-1} \|v\|^2_{0,K} + h_K \left\|\nabla v\right\|^2_{0,K}, \quad \forall v \in H^1(K), \forall K\in \mathcal Q.
\end{equation}
Let us also recall the following local trace inequality proven in \cite{guzman} under the very weak assumption that $\Omega$ is Lipschitz (a proof under stronger assumptions can also be found in \cite{hansbo2} and in \cite{hansbo2bis}): there exists a fixed $h_0>0$ such that if for all $K\in \mathcal Q$ such that \review{$\gamma_K\neq \emptyset$}, $h_K<h_0$, then 
\begin{equation} \label{eq:tracegamma}
\|v\|^2_{0,\gamma_K} \lesssim h_K^{-1} \|v\|^2_{0,K} + h_K \left\|\nabla v\right\|^2_{0,K}, \quad \forall v \in H^1(K),
\end{equation}
where the hidden constant is independent \review{of} $v$, $K$, $h$ and \review{of} how $\gamma$ intersects $K$.

Furthermore, since the hierarchical mesh $\mathcal Q$ is $\mathcal T$-admissible, it is possible to build a Scott-Zhang-type operator $I_h : H_{0,\Gamma_D}^1(\Omega_0) \to V_h^0(\Omega_0)$ on $\Omega_0$ such that for all $v\in H_{0,\Gamma_D}^1(\Omega_0)$, 
\begin{align}
&\sum_{K\in\mathcal Q} h_K^{-2} \|v-I_h(v)\|^2_{0,K} \lesssim \|\nabla v\|^2_{0,\Omega_0} \quad 
\text{ and }\quad \sum_{K\in\mathcal Q} \|\nabla I_h(v)\|^2_{0,K} \lesssim \|\nabla v\|^2_{0,\Omega_0}.\label{eq:L2H1scottzhang}
\end{align}
Note that the latter implies that for all $v\in H_{0,\Gamma_D}^1(\Omega_0)$, 
\begin{align}
\sum_{K\in\mathcal Q} \left\|\nabla \big( v-I_h(v) \big)\right\|^2_{0,K} \lesssim \|\nabla v\|^2_{0,\Omega_0}, \label{eq:H1H1scottzhang}
\end{align}
and that all the quantities considered here refer to the non-trimmed geometry $\Omega_0$. 
\review{The construction of such operator can be found in \cite{corrigendum} and \cite[Section~6.1.3]{reviewadaptiveiga} in the case in which $\Gamma_D = \partial \Omega$, but it readily generalizes to the case in which $\Gamma_D$ is a union of full faces.}\\

Finally, we will need \review{the} two following lemmas to take care of the trimmed elements and faces.
\begin{lemma} \label{lemma:savareinterior}
	Let $K\in \mathcal Q_\trim$, and let $K_\Omega := K\cap \Omega\neq \emptyset$. Then for all $v\in H^1(K)$, 
	$$\| v \|_{0,K_\Omega} \lesssim c_{K_\Omega} \left|K_\Omega\right|^\frac{1}{n}\left( h_K^{-2} \|v\|^2_{0,K} + \|\nabla v\|^2_{0, K}\right)^\frac{1}{2} = \delta_K\left( h_K^{-2} \|v\|^2_{0,K} + \|\nabla v\|^2_{0, K}\right)^\frac{1}{2},$$
	where $c_{K_\Omega}$ is defined as in (\ref{eq:constant}). The hidden constant is independent \review{of} the measures of $K_\Omega$ and of $K$. 
\end{lemma}
\begin{proof}
	Let $v\in H^1(K)$ and let \review{$\mathbf F_K := \mathbf F \circ \mathbf G_K : \hat K \to K$ where $\hat K :=(0,2\pi)^n$,} $\mathbf F$ is the isogeometric map, and $\mathbf G_K$ is a linear mapping. So the Jacobian matrix of $\mathbf G_K$ written $D\mathbf G_K$ is diagonal, and thanks to the shape regularity assumption (A1), each of its components $\left(D\mathbf G_K\right)_{ii}\lesssim h_K$, $i=1,\ldots,n$. Thanks to assumption (A0) on $\mathbf F$, if $D\mathbf F$ denotes the Jacobian matrix of $\mathbf F$, then $\left|\det\left(D\mathbf F\right)\right| \simeq 1$, 
	$\|D\mathbf F\|_{L^\infty\left(\mathbf F^{-1}\left(K\right)\right)} \lesssim 1$, and $\|D\mathbf F^{-1}\|_{L^\infty(K)} \lesssim 1$.
	So if we let $\hat v := v\circ \mathbf F_K$, then by \review{the} H\"older inequality, for all $p\geq 1$, 
	\begin{equation} \label{eq:averageerrorint}
	\|v \|^2_{0,K_\Omega} \leq \left|K_\Omega\right|^{1-\frac{1}{p}} \|v\|^2_{L^{2p}\left(K_\Omega\right)} \leq \left|K_\Omega\right|^{1-\frac{1}{p}} \|v\|^2_{L^{2p}(K)} \lesssim \left|K_\Omega\right|^{1-\frac{1}{p}} h_K^\frac{n}{p} \,\|\hat v\|^2_{L^{2p}\left(\hat K\right)}.
	\end{equation}
	By Sobolev embedding, it is well known that $H^1\left(\hat K\right)$ can be continuously embedded in $L^{2p}\left(\hat K\right)$ for every $1\leq p<\infty$ if $n=2$, or for every $1\leq p\leq 3$ if $n=3$.
	So, if $n=3$, by taking $p=3$ in (\ref{eq:averageerrorint}) and by Sobolev embedding, 
	\begin{align*} 
	\|v \|_{0,K_\Omega}^2 \lesssim |K_\Omega|^{\frac{2}{3}} h_K \left\|\hat v\right\|^2_{1,\hat K} &\lesssim |K_\Omega|^{\frac{2}{3}} h_K \left( h_K^{-3}\left\|v\right\|^2_{0, K} + h_K^{-1} \|\nabla v\|^2_{0, K} \right) \\ &= c_{K_\Omega}^2\left|K_\Omega\right|^\frac{2}{n} \left( h_K^{-2}\left\|v\right\|^2_{0, K} + \|\nabla v\|^2_{0, K} \right).
	\end{align*}
	Let us now consider the case $n=2$. From Appendix \ref{lemma:LqH1},
	$\|\hat v\|_{L^{2p}\left(\hat K\right)} \leq c \sqrt{p} \|\hat v\|_{1,\hat K}$, where $c$ is a constant independent \review{of} $p$ and \review{of} the measure of $K$. So by taking $p=\max\big(\hspace{-0.08cm}-\log\left(\left|K_\Omega\right|\right), \eta \big) = c_{K_\Omega}^2$ in (\ref{eq:averageerrorint}), then $\left|K_\Omega\right|^{-\frac{1}{p}} \lesssim 1$ and $h_K^\frac{2}{p} \simeq |K|^\frac{1}{p} = \left|K\right|^{\displaystyle c_{K_\Omega}^{-2}} \leq \left|K\right|^{\displaystyle c_{K}^{-2}} \lesssim 1$, and thus
	\begin{align*} 
	\| v \|_{0,K_\Omega}^2 &\lesssim p \left|K_\Omega\right|^{1-\frac{1}{p}} h_K^{\frac{2}{p}} \left\|\hat v\right\|_{1,\hat K}^2 \lesssim c_{K_\Omega}^2 \left|K_\Omega\right| \left( h_K^{-2}\|v\|^2_{0,K} + \|\nabla v\|^2_{0, K} \right) = c_{K_\Omega}^2\left|K_\Omega\right|^\frac{2}{n} \left( h_K^{-2}\|v\|^2_{0,K} + \|\nabla v\|^2_{0, K} \right). 
	\end{align*}
\end{proof}

\begin{lemma} \label{lemma:savareboundary}
	Let $F\in \mathcal F_\trim$, and let $F_N := F\cap \Gamma_N\neq \emptyset$. Then for all $v\in H^\frac{1}{2}(F)$,
	$$\| v \|_{0,F_N} \lesssim c_{F_N} \left|F_N\right|^\frac{1}{2(n-1)}\left( h_F^{-1} \|v\|^2_{0,F} + |v|^2_{\frac{1}{2}, F}\right)^\frac{1}{2} = \delta_F\left( h_F^{-1} \|v\|^2_{0,F} + |v|^2_{\frac{1}{2}, F}\right)^\frac{1}{2},$$
	where $c_{F_N}$ is defined as in (\ref{eq:constant}). The hidden constant is independent \review{of} the measures of $F_N$ and of $F$. 
\end{lemma}
\begin{proof} Note that this proof generalizes the one in \cite[Lemma~A.2]{paper1defeaturing}, and it follows the same ideas as in the proof of Lemma \ref{lemma:savareinterior}.
	Let $v\in H^\frac{1}{2}(F)$ and let \review{$\mathbf F_F := \mathbf F \circ \mathbf G_F : \hat F \to F$} where $\hat F := (0,2\pi)^{n-1}$, $\mathbf F$ is the isogeometric map, and $\mathbf G_F$ is a linear mapping. So the Jacobian matrix of $\mathbf G_F$ written $D\mathbf G_F$ is diagonal, and thanks to the shape regularity assumption (A1), each of its components $\left(D\mathbf G_F\right)_{ii}\lesssim h_F$, $i=1,\ldots,n-1$. Thanks to assumption (A0) on $\mathbf F$, if $D\mathbf F$ denotes the Jacobian matrix of $\mathbf F$, then $\left|\det\left(D\mathbf F\right)\right| \simeq 1$, 
	$\|D\mathbf F\|_{L^\infty\left(\mathbf F^{-1}\left(\review{F}\right)\right)} \lesssim 1$, and $\|D\mathbf F^{-1}\|_{L^\infty(\review{F})} \lesssim 1$.
	So if we let $\hat v := v\circ \mathbf F_F$, then by \review{the} H\"older inequality, for all $p\geq 1$, 
	\begin{equation} \label{eq:averageerror}
	\|v \|^2_{0,F_N} \leq \left|F_N\right|^{1-\frac{1}{p}} \|v\|^2_{L^{2p}\left(F_N\right)} \leq \left|F_N\right|^{1-\frac{1}{p}} \|v\|^2_{L^{2p}(F)} \lesssim \left|F_N\right|^{1-\frac{1}{p}} h_F^\frac{n-1}{p} \,\|\hat v\|^2_{L^{2p}\left(\hat F\right)}.
	\end{equation}
	By Sobolev embedding, it is well known that $H^\frac{1}{2}\left(\hat F\right)$ can be continuously embedded in $L^{2p}\left(\hat F\right)$ for every $1\leq p<\infty$ if $n=2$, or for every $1\leq p\leq 2$ if $n=3$.
	So, if $n=3$, by taking $p=2$ in (\ref{eq:averageerror}) and by Sobolev embedding, 
	\begin{equation*} 
	\|v \|_{0,F_N}^2 \lesssim |F_N|^{\frac{1}{2}} h_F \left\|\hat v\right\|^2_{\frac{1}{2},\hat F} \lesssim |F_N|^{\frac{1}{2}} h_F \left( h_F^{-3}\left\|v\right\|^2_{0, F} + h_F^{-1} |v|^2_{\frac{1}{2},F} \right) = c_{F_N}^2\left|F_N\right|^\frac{1}{n-1} \left( h_F^{-2}\left\|v\right\|^2_{0, F} + |v|^2_{\frac{1}{2},F} \right).
	\end{equation*}
	Let us now consider the case $n=2$. From Appendix \ref{lemma:savare},
	\review{$\|\hat v\|_{L^{2p}\left(\hat F\right)} \leq \review{c} \sqrt{p} \|\hat v\|_{\frac{1}{2},\hat F}$, where $c$ is a constant independent {of} $p$ and {of} the measure of $F$.} 
	So by taking $p=\max\big(\hspace{-0.08cm}-\log\left(\left|F_N\right|\right), \eta \big) = c_{F_N}^2$ in (\ref{eq:averageerror}), then $\left|F_N\right|^{-\frac{1}{p}} \lesssim 1$ and $h_F^\frac{1}{p} = \left|F\right|^{c_{F_N}^{-2}} \leq \left|F\right|^{c_{F}^{-2}} \lesssim 1$, and thus
	\begin{align*} 
	\| v \|_{0,F_N}^2 &\lesssim p \left|F_N\right|^{1-\frac{1}{p}} h_F^{\frac{1}{p}} \left\|\hat v\right\|_{\frac{1}{2},\hat F}^2 \lesssim c_{F_N}^2 \left|F_N\right| \left( h_F^{-2}\|v\|^2_{0,F} + |v|^2_{\frac{1}{2},F} \right) = c_{F_N}^2\left|F_N\right|^\frac{1}{n-1} \left( h_F^{-2}\|v\|^2_{0,F} + |v|^2_{\frac{1}{2},F} \right). 
	\end{align*}
\end{proof}

Let us now state and prove the main theorem. 
\begin{theorem} \label{thm:upperbound}
	Let $u$ be the exact solution of problem (\ref{eq:weakpb}) in the trimmed geometry $\Omega$, let $u_h$ be its discretized counterpart that solves problem (\ref{eq:discrpb}), and assume that the mesh $\mathcal Q$ satisfies the assumptions of shape regularity (A1) and $\mathcal T$-admissibility (A2) with $h<h_0$ for some \review{fixed} $h_0\review{>0}$. Then the numerical error, in \review{the} energy norm, is bounded in terms of the estimator $\mathscr{E}(u_h)$ introduced in (\ref{eq:estimator}) as follows:
	$$\|\nabla (u-u_h)\|_{0,\Omega} \lesssim \mathscr{E}(u_h).$$
\end{theorem}

\begin{proof}
	Let $e:= u-u_h$. Then for all $v\in H_{0,\Gamma_D}^1(\Omega)$, from (\ref{eq:weakpb}), 
	\begin{equation} \label{eq:errorint1}
	\int_\Omega \nabla e \cdot \nabla v \,\mathrm dx = \sum_{K\in\mathcal Q_\Omega} \left[\int_{K\cap\Omega} (f+\Delta u_h) v \,\mathrm dx + \int_{\Gamma_N\cap\overline{K}} \left(\review{g_N}-\frac{\partial u_h}{\partial\mathbf{n}}\right) v \,\mathrm ds\right].
	\end{equation}
	Let us define
	\begin{equation} \label{eq:notation}
	\begin{cases}
	r_K:= f+\Delta u_h\in L^2(K\cap\Omega), &\forall K\in \mathcal{Q}_\Omega, \\
	j_F := \review{g_N}-\displaystyle\frac{\partial u_h}{\partial \mathbf n} \in L^2(F\cap\Gamma_N), & \forall F\in\mathcal F_N, \vspace{0.1cm}\\
	j_K:= \review{g_N}-\displaystyle\frac{\partial u_h}{\partial \mathbf n} \in L^2(\gamma_K), & \forall K\in\mathcal Q_\trim.
	\end{cases}\end{equation}
	So with the particular choice $v=e$, we can rewrite (\ref{eq:errorint1}) as
	\begin{align} \label{eq:errorint}
	\|\nabla e\|^2_{0,\Omega} = &\sum_{K\in\mathcal Q_\Omega} \int_{K\cap\Omega} r_K e \,\mathrm dx + \sum_{F\in\mathcal F_N} \int_{F\cap\Gamma_N} j_F\,e\,\mathrm ds + \sum_{K\in\mathcal Q_\trim} \int_{\gamma_K} j_K e \,\mathrm ds.
	\end{align}
	
	Moreover, recall that $D := \Omega_0\setminus\overline\Omega$ (see Figure \ref{fig:domains}). Then, let ${\mathbf n}_D$ be the outward unit normal to $\partial D$, and let $e_0\in H^1_{0,\Gamma_D}(\Omega_0)$ be the extension of $e$ to $\Omega_0$ such that $e_0\vert_{\Omega} = e$, and $e_0\vert_{D}\in H^1_{e,\gamma}(D)$ is the weak solution of
	$$\int_D \nabla e_0 \cdot \nabla w \,\mathrm dx = 0, \qquad \forall w \in H_{0,\gamma}^1(D).$$
	Since the measures of both $D$ and $\gamma$ are independent \review{of} $h$, then by continuity of the elliptic solution on the data and by the trace inequality, 
	\begin{equation*} 
	\|e_0\|_{1,D} \lesssim \|e\|_{\frac{1}{2},\gamma} \lesssim \|e\|_{1,\Omega}, 
	\end{equation*}
	and thus by \review{the} Poincar\'e inequality, 
	\begin{equation}\label{eq:equive0e}
	\left\| \nabla e_0 \right\|_{0,\Omega_0} \lesssim \left(\| e_0\|_{1,\Omega}^2 + \| e_0\|_{1,D}^2\right)^\frac{1}{2}\lesssim \|e\|_{1,\Omega} \lesssim \|\nabla e\|_{0,\Omega}.
	\end{equation}
	Furthermore, let $I_h : H_{0,\Gamma_D}^1(\Omega_0) \to V_h^0(\Omega_0)$ be the Scott-Zhang-type operator on $\Omega_0$ as introduced in (\ref{eq:L2H1scottzhang}), and let $e_{0h}:= I_h(e_0)\in V_h(\Omega_0)$. 
	Since $e_{0h}\in V_h(\Omega_0)$, from (\ref{eq:discrpb}) and by performing integration by parts, 
	\begin{align*}
	\int_\Omega fe_{0h} \,\mathrm dx + \int_{\Gamma_N} \review{g_N} e_{0h} \,\mathrm ds = \int_\Omega \nabla u_h \cdot \nabla e_{0h} \,\mathrm dx = \sum_{K\in\mathcal Q_\Omega}\int_{K\cap\Omega} \left(-\Delta u_h\right)e_{0h} \,\mathrm dx + \int_{\Gamma_N} \frac{\partial u_h}{\partial \mathbf n} e_{0h} \,\mathrm ds,
	\end{align*}
	and thus recalling the notation introduced in (\ref{eq:notation}),
	\begin{align*}
	\sum_{K\in\mathcal Q_\Omega} \int_{K\cap\Omega} r_K e_{0h} \,\mathrm dx + \sum_{F\in\mathcal F_N} \int_{F\cap\Gamma_N} j_F\,e_{0h}\,\mathrm ds + \sum_{K\in\mathcal Q_\trim} \int_{\gamma_K} j_K e_{0h} \,\mathrm ds=0. 
	\end{align*}
	Consequently, we can rewrite (\ref{eq:errorint}) as
	\begin{equation} \label{eq:beforestd}
	\|\nabla e\|^2_{0,\Omega} = \sum_{K\in\mathcal Q_\Omega} \int_{K\cap\Omega} r_K (e-e_{0h}) \,\mathrm dx + \sum_{F\in\mathcal F_N} \int_{F\cap\Gamma_N} j_F(e-e_{0h})\,\mathrm ds + \sum_{K\in\mathcal Q_\trim} \int_{\gamma_K} j_K (e-e_{0h}) \,\mathrm ds.
	\end{equation}
	To begin with, let us consider the first term. Using \review{the} H\"older inequality and the discrete Cauchy-Schwarz inequality, 
	\begin{align}
	&\sum_{K\in\mathcal Q_\Omega} \int_{K\cap\Omega} r_K(e_0-e_{0h})\,\mathrm dx \nonumber\\
	\leq &\sum_{K\in\mathcal Q_\Omega} \delta_K \|r_K\|_{0,K\cap\Omega} \,\delta_K^{-1} \|e_0-e_{0h}\|_{0,K\cap\Omega} \nonumber \\
	\lesssim &\left( \sum_{K\in\mathcal Q_\Omega} \delta_K^2 \|r_K\|^2_{0,K\cap\Omega}\right)^{\hspace{-0.1cm}\frac{1}{2}}\hspace{-0.2cm}\left(\sum_{K\in\mathcal Q_\untr} \delta_K^{-2} \|e_0-e_{0h}\|^2_{0,K} + \sum_{K\in\mathcal Q_\trim} \delta_K^{-2} \|e_0-e_{0h}\|^2_{0,K\cap\Omega} \right)^{\hspace{-0.1cm}\frac{1}{2}}\hspace{-0.1cm}. \label{eq:beforeclemppties}
	\end{align}
	Moreover, from Lemma \ref{lemma:savareinterior}, for all $K\in \mathcal Q_\trim$, 
	\begin{equation} \label{eq:alphaKineq}
	\delta_K^{-2} \|e_0-e_{0h}\|^2_{0,K\cap\Omega} \lesssim h_K^{-2}\|e_0-e_{0h}\|^2_{0,K} + \|\nabla(e_0-e_{0h})\|^2_{0,K}. 
	\end{equation}
	Thus by (\ref{eq:alphaKineq}), using property (\ref{eq:L2H1scottzhang}) of the Scott-Zhang-type operator \review{which is applicable} since $\mathcal{Q}$ is $\mathcal T$-admissible, and by the continuous extension property (\ref{eq:equive0e}), 
	\begin{align} \label{eq:intresclem}
	&\sum_{K\in\mathcal Q_\untr} \delta_K^{-2} \|e_0-e_{0h}\|^2_{0,K} + \sum_{K\in\mathcal Q_\trim} \delta_K^{-2} \|e_0-e_{0h}\|^2_{0,K\cap\Omega} \nonumber \\
	\lesssim &\sum_{K\in\review{\mathcal Q_\Omega}} h_K^{-2} \|e_0-e_{0h}\|^2_{0,K} + \sum_{K\in\mathcal Q_\trim} \|\nabla(e_0-e_{0h})\|^2_{0,K} \lesssim \|\nabla e_0\|^2_{0,\Omega_0} \lesssim \|\nabla e\|^2_{0,\Omega}. 
	\end{align}
	\review{Therefore}, combining (\ref{eq:beforeclemppties}) and (\ref{eq:intresclem}), 
	\begin{equation} \label{eq:internalresidualterm}
	\sum_{K\in\mathcal Q_\Omega} \int_{K\cap\Omega} r_K(e_0-e_{0h})\,\mathrm dx \lesssim \left( \sum_{K\in\mathcal Q_\Omega}\delta_K^2 \|r_K\|^2_{0,K\cap\Omega}\right)^\frac{1}{2} \|\nabla e\|^2_{0,\Omega}.
	\end{equation}
	
	Now, let us consider the second term of (\ref{eq:beforestd}). Similarly as for the internal residuals, 
	\begin{align}
	&\sum_{F\in\mathcal F_N} \int_{F\cap\Gamma_N} j_F(e-e_{0h})\,\mathrm ds \leq \sum_{F\in\review{\mathcal F_N}} \delta_F\left\|j_F\right\|_{0,F\cap\Gamma_N}\delta_F^{-1}\|e_0-e_{0h}\|_{0,F\cap\Gamma_N} \nonumber \\
	\lesssim &\left( \sum_{F\in\mathcal F_N}\delta_F^2 \|j_F\|^2_{0,F\cap\Gamma_N}\right)^\frac{1}{2} \hspace{-0.1cm}\left(\sum_{F\in\mathcal F_\untr} \delta_F^{-2} \|e_0-e_{0h}\|^2_{0,F} + \sum_{F\in\mathcal F_\trim} \delta_F^{-2} \|e_0-e_{0h}\|^2_{0,F\cap\Gamma_N} \right)^\frac{1}{2}\hspace{-0.1cm}. \label{eq:beforeclempptiesbd}
	\end{align}
	Moreover, from Lemma \ref{lemma:savareboundary}, for all $F\in \mathcal F_\trim$, 
	\begin{equation*} 
	\delta_F^{-2} \|e_0-e_{0h}\|^2_{0,F\cap\Gamma_N} \lesssim h_F^{-1}\|e_0-e_{0h}\|^2_{0,F} + \left|e_0-e_{0h}\right|^2_{\frac{1}{2},F}, 
	\end{equation*}
	and thus 
	\begin{equation} \label{eq:2partsbd}
	\sum_{F\in\mathcal F_\untr} \delta_F^{-2} \|e_0-e_{0h}\|^2_{0,F} + \sum_{F\in\mathcal F_\trim} \delta_F^{-2} \|e_0-e_{0h}\|^2_{0,F\cap\Gamma_N} \lesssim \sum_{F\in\review{\mathcal F_N}} h_F^{-1} \|e_0-e_{0h}\|^2_{0,F} + \sum_{F\in\mathcal F_\trim} |e_0-e_{0h}|^2_{\frac{1}{2},F}.
	\end{equation}
	For all $F\in\review{\mathcal F_N}$, let $K_F$ be the element of $\mathcal{Q}$ such that $\left(F\cap\Gamma_N\right) \subset \partial (K_F\cap\Omega)$, and note that by the shape regularity of $\mathcal Q$, $h_{K_F}\simeq h_{F}$. Then for the first term of (\ref{eq:2partsbd}), using the scaled trace inequality (\ref{eq:trace}), properties (\ref{eq:L2H1scottzhang}) and (\ref{eq:H1H1scottzhang}) of the Scott-Zhang-type operator, and by the continuous extension property (\ref{eq:equive0e}), 
	\begin{align}
	\sum_{F\in\review{\mathcal F_N}} h_F^{-1} \|e_0-e_{0h}\|^2_{0,F} &\lesssim \sum_{F\in\review{\mathcal F_N}} \big( h_{K_F}^{-2}\|e_0-e_{0h}\|^2_{0,K_F} + \|\nabla (e_0-e_{0h})\|^2_{0,K_F} \big) \lesssim \|\nabla e_0\|^2_{0,\Omega_0} \lesssim \|\nabla e\|^2_{0,\Omega}. \label{eq:bdry1}
	\end{align}
	\review{Furthermore, remarking that $e_0 - e_{0h} \in H_{0,\Gamma_D}^1(\Omega_0)$, then the second term of~(\ref{eq:2partsbd}) can be estimated using trace inequalities, property~(\ref{eq:H1H1scottzhang}) of the Scott-Zhang-type operator, and the continuous extension property~(\ref{eq:equive0e}) as follows:
		\begin{align}
		\sum_{F\in\mathcal F_\trim} |e_0-e_{0h}|^2_{\frac{1}{2},F} &\lesssim \sum_{F\in\mathcal F_\trim} \|e_0-e_{0h}\|^2_{1, K_F} \leq \| e_0 - e_{0h} \|^2_{1,\Omega_0} \nonumber \\
		& \lesssim \|\nabla( e_0-e_{0h} )\|^2_{0, \Omega_0} \lesssim \|\nabla e_0\|^2_{0,\Omega_0} \lesssim \|\nabla e\|_{0,\Omega}^2. \label{eq:bdry2}
		\end{align}}
	Therefore, combining (\ref{eq:beforeclempptiesbd}), (\ref{eq:2partsbd}), (\ref{eq:bdry1}) and (\ref{eq:bdry2}), 
	\begin{equation}
	\sum_{F\in\mathcal F_N} \int_{F\cap\Gamma_N} j_F(e-e_{0h})\,\mathrm ds \lesssim \left( \sum_{F\in\mathcal F_N} \delta_F^2 \|j_F\|^2_{0,F\cap\Gamma_N}\right)^\frac{1}{2} \|\nabla e\|_{0,\Omega}. \label{eq:boundaryterm}
	\end{equation}
	
	Finally, let us consider the last term of (\ref{eq:beforestd}): using \review{the} H\"older inequality and the discrete Cauchy-Schwarz inequality, then
	\begin{align}
	\sum_{K\in\mathcal Q_\trim} \int_{\gamma_K} j_K (e_0-e_{0h})\,\mathrm ds &\leq \sum_{K\in\mathcal Q_\trim} \|j_K\|_{0,\gamma_K}\|e_0-e_{0h}\|_{0,\gamma_K} \nonumber\\
	&\lesssim \left(\sum_{K\in\mathcal Q_\trim} h_K \|j_K\|^2_{0,\gamma_K}\right)^\frac{1}{2} \left( \sum_{K\in\mathcal Q_\trim} h_K^{-1}\|e_0-e_{0h}\|_{0,\gamma_K}^2\right)^\frac{1}{2}. \label{eq:estgammaK}
	\end{align}
	Moreover, by the local scaled trace inequality (\ref{eq:tracegamma}), by properties (\ref{eq:L2H1scottzhang}) and (\ref{eq:H1H1scottzhang}) of the Scott-Zhang-type operator, and by the continuous extension property (\ref{eq:equive0e}), 
	\begin{align}
	\sum_{K\in\mathcal Q_\trim} h_K^{-1}\|e_0-e_{0h}\|_{0,\gamma_K}^2 
	&\lesssim \sum_{K\in\mathcal Q_\trim} \left( h_K^{-2}\|e_0-e_{0h}\|^2_{0,K} + \|\nabla (e_0-e_{0h})\|^2_{0,K} \right) \lesssim \|\nabla e_0\|^2_{0,\Omega_0} \lesssim \|\nabla e\|_{0,\Omega}^2. \label{eq:estgammaK2}
	\end{align}
	Therefore, inserting (\ref{eq:estgammaK2}) in (\ref{eq:estgammaK}), we obtain
	\begin{align}
	\sum_{K\in\mathcal Q_\trim} \int_{\gamma_K} j_K (e_0-e_{0h})\,\mathrm ds &\lesssim \left(\sum_{K\in\mathcal Q_\trim} h_K \|j_K\|^2_{0,\gamma_K}\right)^\frac{1}{2} \|\nabla e\|_{0,\Omega}. \label{eq:lastterm}
	\end{align} 
	
	To conclude, we combine (\ref{eq:beforestd}), (\ref{eq:internalresidualterm}), (\ref{eq:boundaryterm}) and (\ref{eq:lastterm}) and we \review{divide by} $\|\nabla e\|_{0,\Omega}$ on both sides.
	
\end{proof}

\begin{remark}
	If problem (\ref{eq:weakpb}) is discretized using the THB-spline basis counterpart of the HB-spline basis defined in (\ref{eq:thbsplinebasis}), then Theorem \ref{thm:upperbound} is still valid under the same hypothesis. The interested reader is referred to \cite{vuong,thbgiannelli} for more details about THB-splines.
	In particular, whether HB-splines or THB-splines are used, the mesh $\mathcal Q$ only needs to be $\mathcal T$-admissible (constraint coming from the study of THB-splines, see Definition \ref{def:admissibility}), which is a weaker assumption than the $\mathcal H$-admissibility defined for example in \cite[Section~4.1.3]{reviewadaptiveiga} (which is similarly coming from the study of HB-splines, see \cite{gantneradaptiveiga}). Indeed, in the proof of Theorem \ref{thm:upperbound}, the $\mathcal T$-admissibility is only needed to ensure properties (\ref{eq:L2H1scottzhang}) and (\ref{eq:H1H1scottzhang}) of the Scott-Zhang-type operator $I_h: H_{0,\Gamma_D}^1(\Omega_0) \to V_h^0(\Omega_0)$. The construction of such \review{an} operator depends on the discrete space $V_h^0(\Omega_0)$ but not on the choice of basis for $V_h^0(\Omega_0)$, so it does not depend on whether one considers an HB-spline basis or its truncated counterpart. However, the sparsity and the conditioning of the matrices involved in the computation depend on this choice.
\end{remark}

\review{\begin{remark} \label{rmk:extC0}
		The proof of Theorem~\ref{thm:upperbound} can be readily extended to the setting of $C^0$-continuous trimmed (T)HB-spline basis functions through the introduction of appropriate jump terms, as the proof only relies on discretization-independent trace inequalities and on the existence of a Scott-Zhang-type interpolation operator with properties~(\ref{eq:L2H1scottzhang}) and (\ref{eq:H1H1scottzhang}). 
	\end{remark}
	
	\begin{remark} \label{rmk:extweakDirichlet}
		Similar arguments could be used to extend the proof of Theorem~\ref{thm:upperbound} to the more general case in which Dirichlet boundary conditions are also applied along the trimming boundary $\Gamma^\mathrm t$. Indeed, in this case, Dirichlet boundary conditions need to be enforced weakly, using for instance Nitsche’s method, but then the discrete problem needs to be stabilized, following for instance \cite{ghostpenalty,BURMAN2012328,stabcutiga,puppistabilization}. However, some work remains to be done because special care must to be given to the extra terms coming from the stabilized Nitsche's method, as they contain integrals over the trimmed Dirichlet boundary. 
	\end{remark}
	
	\begin{remark}
		Theorem~\ref{thm:upperbound} states the reliability of the proposed error estimator but it does not state its efficiency, as the efficiency proof presents some further challenges. For instance, classical efficiency proofs on non-trimmed domains (see, e.g., \cite{verfurth1994posteriori,nochettoprimer,buffagiannelli1}) make use of cut-off polynomial functions to localize the error in one element or in a patch of elements. Then, inverse inequalities are used to bound each estimator's residual contribution by a local error term. However, polynomial cut-off functions, sometimes also called bubble functions, cannot be defined in the (non-necessarily polygonal) active part of a trimmed element. Another path one could take is to extend each trimmed element residual to the full non-trimmed element, by considering the natural polynomial extension of the discrete solution and an $L^2$-extension of the source and Neumann boundary functions. Then, one could perform the same standard steps of the efficiency proof by considering cut-off functions in the non-trimmed elements. But in this case, the wrong element scalings would be obtained if $|K\cap\Omega| \ll |K|$ for an element $K\in\mathcal Q_\trim$. 
\end{remark}}

%% file: adaptivity.tex
Closely following the framework of adaptive finite elements found in \cite{nochettoprimer} for elliptic partial differential equations, and \cite{trimmedshells} in the context of hierarchical trimmed geometries, we recall the four main building blocks of adaptivity, composing one iteration of the iterative process:
\begin{figure}[h!]
	\begin{center}
		\begin{tikzpicture}[
		start chain = going right,
		node distance=7mm,
		block/.style={shape=rectangle, draw,
			inner sep=1mm, align=center,
			minimum height=5mm, minimum width=15mm, on chain}] 
		\node[block] (n1) {SOLVE};
		\node[block] (n2) {ESTIMATE};
		\node[block] (n3) {MARK};
		\node[block] (n4) {REFINE};
		
		\draw[->] (n1.east) --  + (0,0mm) -> (n2.west);
		\draw[->] (n2.east) --  + (0,0mm) -> (n3.west);
		\draw[->] (n3.east) --  + (0,0mm) -> (n4.west);
		\draw[<-] (n1.south) --  + (0,-5mm) -| (n4.south);
		\end{tikzpicture}
	\end{center}
\end{figure}

The first two modules, SOLVE and ESTIMATE, have been elaborated in Sections \ref{sec:modelpb} and \ref{sec:estimator}, respectively. That is, we approximate the exact solution $u$ of problem (\ref{eq:originalpb}) using the Galerkin method with finite basis $\mathcal H_\Omega$, which is the HB-spline basis defined on a $\mathcal T$-admissible mesh $\mathcal Q$, in the trimmed geometry $\Omega$, as introduced in (\ref{eq:HOmegabasis}). We obtain the discrete solution $u_h$ solving problem (\ref{eq:discrpb}). Then, we can estimate the energy error $\|\nabla (u-u_h)\|_{0,\Omega}$ thanks to the estimator $\mathscr{E}(u_h)$ defined in (\ref{eq:estimator}), using Theorem \ref{thm:upperbound}. 

Remark that $\mathscr E(u_h)$ is naturally decomposed into local element contributions as follows. \review{If we} let $$\mathcal F_N^K := \left\{ F\in \mathcal F_N : \left(F\cap\Gamma_N\right) \subset \partial (K\cap\Omega) \right\},$$
then the estimator rewrites as
\begin{align*}
\mathscr E(u_h) &:= \left[ \sum_{K\in\mathcal Q} \mathscr{E}_K(u_h)^2\right]^\frac{1}{2},
\end{align*}
where for all $K\in \review{\mathcal Q}$, 
\begin{align*}
\mathscr E_K(u_h)^2 := \begin{cases}
\delta_K^2 \left\| f+\Delta u_h\right\|_{0,K}^2 + \displaystyle\sum_{F\in\mathcal F_N^K} \delta_F^2 \left\|\review{g_N}-\frac{\partial u_h}{\partial \mathbf n}\right\|^2_{0,F\cap\Gamma_N}, & \forall K\in \mathcal Q_\untr \\
\delta_K^2 \left\| f+\Delta u_h\right\|_{0,K\cap\Omega}^2 + \displaystyle\sum_{F\in\mathcal F_N^K} \delta_F^2 \left\|\review{g_N}-\frac{\partial u_h}{\partial \mathbf n}\right\|^2_{0,F\cap\Gamma_N} + h_K \left\|\review{g_N}-\frac{\partial u_h}{\partial \mathbf n}\right\|^2_{0,\gamma_K}, & \forall K\in \mathcal Q_\trim \\
0 & \text{else,}
\end{cases}
\end{align*}
and $\delta_K$ and $\delta_F$ are defined in (\ref{eq:deltas}). 

Therefore, the module MARK selects a set of elements $\mathcal M \subset \mathcal Q$ according to some marking strategy such as the D\"{o}rfler marking, see \cite{dorfler}, i.e., given a fixed parameter $\theta \in (0,1]$, the set of marked elements $\mathcal M$ satisfies
\begin{equation}\label{eq:dorfler}
\left[ \sum_{K\in\mathcal M} \mathscr{E}_K(u_h)^2\right]^\frac{1}{2} \geq \theta \mathscr{E}(u_h).
\end{equation}
This strategy guarantees that the elements in $\mathcal M$ give a substantial contribution to the total error estimator $\mathscr{E}(u_h)$. Since for all $K\in \mathcal Q$ such that $K\cap \Omega = \emptyset$, $\mathscr{E}_K(u_h) = 0$, then the elements outside of the trimmed domain will never be marked. But \review{from the algorithmic point of view}, to guarantee that a basis function $B\in\mathcal H_\Omega$ is deactivated when all the elements in $\supp{B}\cap\Omega$ are marked for refinement, then we also need to \review{enlarge $\mathcal M$ with} the so-called ghost elements, see \cite{trimmedshells}. That is, when a trimmed element $K$ is marked for refinement, we add all the elements in
$$\review{\left\{ K'\in\mathcal Q: K'\cap\Omega = \emptyset, \, \text{lev}(K)=\text{lev}(K') \text{ and } \exists B\in \mathcal H \text{ such that } K\cup K'\subset\supp{B} \right\}.}$$ 
\review{to the set $\mathcal M$. Note that this is only needed for algorithmic reasons, while it does not change the active refined basis determined by the marking strategy.}

Finally, based \review{on} the set $\mathcal M$ of marked elements, the hierarchical mesh $\mathcal Q$ is refined so that its class of $\mathcal T$-admissibility and its properties are preserved, that is, so that the refined mesh satisfies assumptions (A1) and (A2). Since trimming does not influence the admissibility property of the mesh defined on $\Omega_0$ but only the marking, then the REFINE module is the same as in the non-trimmed case developed in \cite{buffagiannelli1}. For completeness, we summarize the procedure in Algorithms \ref{algo:ref} and \ref{algo:refrec}, using the notion of hierarchical element neighborhood $\mathcal N(\mathcal Q, K,\mu)$ with respect to the admissibility class $\mu$, introduced in Section \ref{ss:hbs}.

\savebox{\algleft}{%
	\begin{minipage}{.49\textwidth}
		\begin{algorithm}[H]
			\caption{\\$\mathcal Q = $ REFINE($\mathcal Q$, $\mathcal M$, $\mu$)}\label{algo:ref}
			\begin{algorithmic}[1]
				\ForAll{$K\in \mathcal Q \cap \mathcal M$}
				\State $\mathcal Q = $ REFINE\_RECURSIVE($\mathcal Q$, $K$, $\mu$)
				\EndFor\\
				\Return $\mathcal Q$
			\end{algorithmic}
		\end{algorithm}
\end{minipage}}
\savebox{\algright}{%
	\begin{minipage}{.5\textwidth}
		\begin{algorithm}[H]
			\caption{\\$\mathcal Q = $ REFINE\_RECURSIVE($\mathcal Q$, $K$, $\mu$)}\label{algo:refrec}
			\begin{algorithmic}[1]
				\ForAll{$K'\in \mathcal N(\mathcal Q, K,\mu)$}
				\State $\mathcal Q = $ REFINE\_RECURSIVE($\mathcal Q$, $K'$, $\mu$)
				\EndFor 
				\If{$K$ has not yet been subdivided}
				\State Subdivide $K$ and update $\mathcal Q$ by replacing $K$ with its children
				\EndIf \\
				\Return $\mathcal Q$
			\end{algorithmic}
		\end{algorithm}
\end{minipage}}
\noindent\usebox{\algleft}\hfill\usebox{\algright}

%% file: numtests.tex
In this last section, we present a few numerical examples to illustrate the validity of the proposed error estimator. All the presented numerical experiments have been implemented in GeoPDEs \cite{geopdes}, an open-source and free Octave/Matlab package for the resolution of partial differential equations specifically designed for isogeometric analysis. Moreover, an in-house tool presented in \cite{antolinvreps} has been used for the geometric description and the local meshing process required for the integration of trimmed geometries. All considered meshes are built and refined to be and remain $\mathcal T$-admissible of class $\mu = p$ as defined in Definition \ref{def:admissibility}, where $p$ is the degree of the B-splines considered in each experiment.

\begin{figure}
	\begin{subfigure}[t]{0.45\textwidth}
		\begin{center}
			\begin{tikzpicture}[scale=4]
			\draw[c5!50] (0.25,0) -- (0.25,1);
			\draw[c5!50] (0.5,0) -- (0.5,1);
			\draw[c5!50] (0.75,0) -- (0.75,1);
			\draw[c5!50] (0,0.25) -- (1,0.25);
			\draw[c5!50] (0,0.5) -- (1,0.5);
			\draw[c5!50] (0,0.75) -- (1,0.75);
			\draw[thick] (0.25,0.25) circle(0.1) ;
			\draw[thick] (0.75,0.75) circle(0.1) ;
			\draw[thick,c1] (0,0) -- (1,0);
			\draw[thick] (0,0) --(0,1) -- (1,1) -- (1,0);
			\draw[thick,c1] (0.5,0) node[below]{$\Gamma_D$};
			\draw (0.25,0.25) node{$D_1$};
			\draw (0.75,0.75) node{$D_2$};
			\draw (0.5,-0.3) node{$ $};
			\end{tikzpicture}
			\caption{Trimmed domain $\Omega$ and original mesh (in gray).} \label{fig:geomregularex}
		\end{center}
	\end{subfigure}
	~ 
	\begin{subfigure}[t]{0.52\textwidth}
		\begin{center}
			\begin{tikzpicture}[scale=0.85]
			\begin{axis}[xmode=log, ymode=log, xlabel=$N_\text{dof}$, ylabel=Error/Estimator, grid, grid style={dotted}]
			\addplot[mark=+, c2, thick, densely dashed, mark options={solid,scale=1.5}] table [x=ndof, y=est, col sep=comma] {data/square_bucato_Tadm.csv};
			\addplot[mark=o, c1, thick, mark options={solid,scale=1}] table [x=ndof, y=err, col sep=comma] {data/square_bucato_Tadm.csv};
			\logLogSlopeReverseTriangle{0.79}{0.1}{0.28}{1}{black}
			\legend{$\mathscr E\left(u_h\right)$, $\left|u-u_h\right|_{1,\Omega}$}; 
			\logLogSlopeReverseTriangle{0.79}{0.1}{0.075}{1}{black}
			\end{axis}
			\end{tikzpicture}
			\caption{Convergence of the energy error and of the estimator with respect to the total number of degrees of freedom. } \label{fig:cvregularex}
		\end{center}
	\end{subfigure}
	\caption{Trimmed geometry and convergence of the error and estimator on a regular problem.} \label{fig:regularex}
\end{figure}
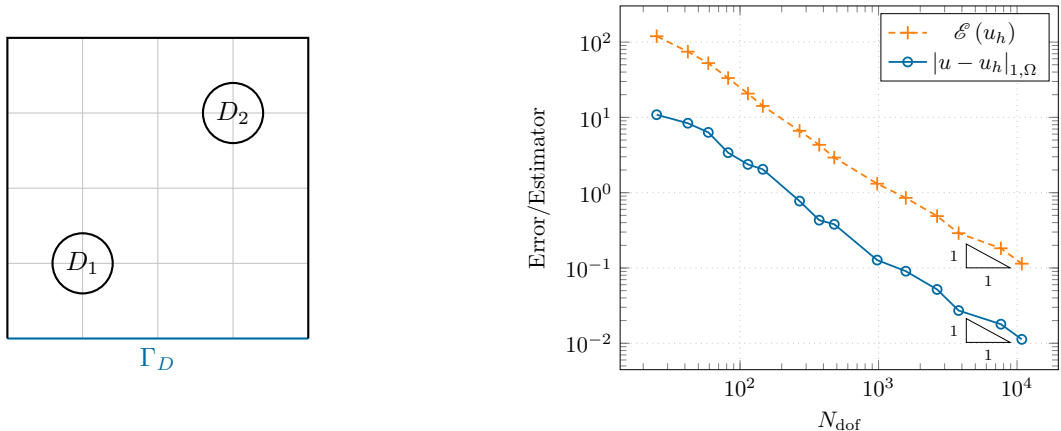

\subsection{Adaptive mesh refinement on a regular solution}
As a first numerical experiment, we choose a regular problem defined in the unit square $\Omega_0=(0,1)^2$ trimmed by two disks $D_1$ and $D_2$ of radius $0.1$ centered in $\left(0.25,0.25\right)$ and $\left(0.75,0.75\right)$, respectively. That is, the trimmed geometry $\Omega = (0,1)^2\setminus \left(\overline{D_1} \cup \overline{D_2}\right)$ is as illustrated in Figure \ref{fig:geomregularex}. We choose the data $f$, $g_D$ and $g_N$ such that the exact solution of Laplace problem (\ref{eq:originalpb}) is $u(x,y) = \sin(3\pi x)+\cos(5\pi y)$, with $\Gamma_D := (0,1)\times \{0\}$ and $\Gamma_N := \partial \Omega \setminus \overline{\Gamma_D}$. 

At the first iteration of the adaptive process presented in Section \ref{sec:adapt}, the non-trimmed geometry $\Omega_0$ is meshed with four elements in each direction, as represented in Figure \ref{fig:geomregularex}. At each iteration, the differential problem is discretized using \review{the} Galerkin method with HB-splines of degree $2$ on the trimmed mesh, as explained in Section \ref{sec:HBIGAformulation}, and the numerical solution $u_h$ of (\ref{eq:discrpb}) is obtained. The energy error $\left\|\nabla(u-u_h)\right\|_{0,\Omega}$ and the estimator (\ref{eq:estimator}) are computed. The latter is then used to drive the mesh refinement strategy described in Algorithm \ref{algo:ref} for which the elements to refine are determined using the D\"orfler marking strategy (\ref{eq:dorfler}) with parameter $\theta = 0.8$. The iterative process is then stopped whenever the number of degrees of freedom of the HB-spline space exceeds $10^4$. 

The results are presented in Figure \ref{fig:cvregularex} and validate the theory presented in Section \ref{sec:estimator}. Indeed, we can observe that the estimator follows the behavior of the energy error, both of them having the same convergence rate of $\mathcal O\left(N_\text{dof}^{-\frac{p}{2}}\right) = \mathcal O\left(N_\text{dof}^{-1}\right)$. The effectivity index is nearly equal to $10$, that is, it is very similar to the effectivity index observed in the case of a residual estimator on non-trimmed geometries, \review{see e.g., \cite{reviewadaptiveiga}}.

\subsection{Independence from the size of the active parts of the trimmed elements: regular solution}
The aim of this second numerical experiment is to verify that the effectivity index defined by $$\eta := \displaystyle\frac{\mathscr{E}(u_h)}{\left\|\nabla(u-u_h)\right\|_{0,\Omega}}$$ is independent \review{of} the measure of the active part of the trimmed elements, that is, \review{of} the measure of $K\cap\Omega$ for every $K\in\mathcal Q$. To verify this, let $\Omega$ be the pentagon obtained by trimming the unit square $\Omega_0 := (0,1)^2$ by the straight line passing through the points $(0,0.25)$ and $(0.75,1)$, as represented in Figure \ref{fig:pentagongeom}. We choose the data $f$, $g_D$ and $g_N$ such that the exact solution of Laplace problem (\ref{eq:originalpb}) is $u(x,y) = \arctan\big(15(x-y+0.25)\big)$, with $\Gamma_D := \big( (0,1)\times \{0\} \big) \cup \big( \{1\}\times (0,1)\big)$ and $\Gamma_N := \partial \Omega \setminus \overline{\Gamma_D}$. 

At the first iteration of the adaptive process, the non-trimmed geometry $\Omega_0$ is meshed with four elements in each direction such that the internal knot lines are the lines defined by $x = \displaystyle\frac{k}{4}+\varepsilon$ and $y = \displaystyle\frac{k}{4}-\varepsilon$ for $k=1,2,3$ and $\varepsilon > 0$ small. In this way, the trimmed mesh has elements whose triangular active part are of measure proportional to $\varepsilon^2$, as shown in Figure \ref{fig:pentagongeom}. We solve Galerkin problem (\ref{eq:discrpb}) with HB-splines of degree $3$ in each direction, with $\varepsilon = 10^{-q}$ for $q=5,6,7$, and we denote $u_h^\varepsilon$ the obtained discrete solution. Note that the dependence on $\varepsilon$ is only on the mesh and not on the geometry $\Omega$ itself. Both uniform refinement and the adaptive refinement described in Section \ref{sec:adapt} with $\theta = 0.9$ are performed. The algorithm stops when the number of degrees of freedom of the HB-spline space exceeds $10^4$. The results are given in Figure \ref{fig:cvpentagonGR} and \ref{fig:cvpentagonGERS}, respectively, \review{and the final mesh obtained by adaptive refinement is represented in Figure~\ref{fig:pentagonfinalmesh}}. 

\begin{figure}
	\begin{subfigure}[t]{0.48\textwidth}
		\begin{center}
			\begin{tikzpicture}[scale=3.5]
			\draw[c5!50] (0.28,0) -- (0.28,1);
			\draw[c5!50] (0.53,0) -- (0.53,1);
			\draw[c5!50] (0.78,0) -- (0.78,1);
			\draw[c5!50] (0,0.22) -- (1,0.22);
			\draw[c5!50] (0,0.47) -- (1,0.47);
			\draw[c5!50] (0,0.72) -- (1,0.72);
			\fill[c4!50] (0.22,0.47) -- (0.28,0.47) -- (0.28,0.53) -- cycle;
			\fill[c4!50] (0.47,0.72) -- (0.53,0.72) -- (0.53,0.78) -- cycle;
			\draw[thick,c1] (0,0) -- (1,0) -- (1,1);
			\draw[thick] (0,0) --(0,0.25);
			\draw[thick] (0.75,1) -- (1,1);
			\draw[thick,c5!50] (0,1) -- (0,0.25);
			\draw[thick,c5!50] (0.75,1) -- (0,1);
			\draw[thick] (0,0.25) -- (0.75,1);
			\draw[thick,c1] (0.9,0) node[below]{$\Gamma_D$};
			\draw[c5!90!black] (0,1) node[left]{\footnotesize $1$};
			\draw[c5!90!black] (0,0.72) node[left]{\footnotesize $\frac{3}{4}-\varepsilon$};
			\draw[c5!90!black] (0,0.47) node[left]{\footnotesize $\frac{1}{2}-\varepsilon$};
			\draw[c5!90!black] (0,0.22) node[left]{\footnotesize $\frac{1}{4}-\varepsilon$};
			\draw[c5!90!black] (0,0) node[left]{\footnotesize $0$};
			\draw [c5!90!black](1,1) node[above]{\footnotesize $1$};
			\draw[c5!90!black] (0.79,1) node[above]{\footnotesize $\frac{3}{4}+\varepsilon$};
			\draw[c5!90!black] (0.53,1) node[above]{\footnotesize $\frac{1}{2}+\varepsilon$};
			\draw[c5!90!black] (0.27,1) node[above]{\footnotesize $\frac{1}{4}+\varepsilon$};
			\draw[c5!90!black] (0,1) node[above]{\footnotesize $0$};
			\end{tikzpicture}
			\caption{Trimmed domain $\Omega$, original mesh (in gray), and highlighted cut elements with active part of measure $\simeq \varepsilon^2$.} \label{fig:pentagongeom}
		\end{center}
	\end{subfigure}
	~ 
	\begin{subfigure}[t]{0.48\textwidth}
		\begin{center}
			\includegraphics[scale=0.103]{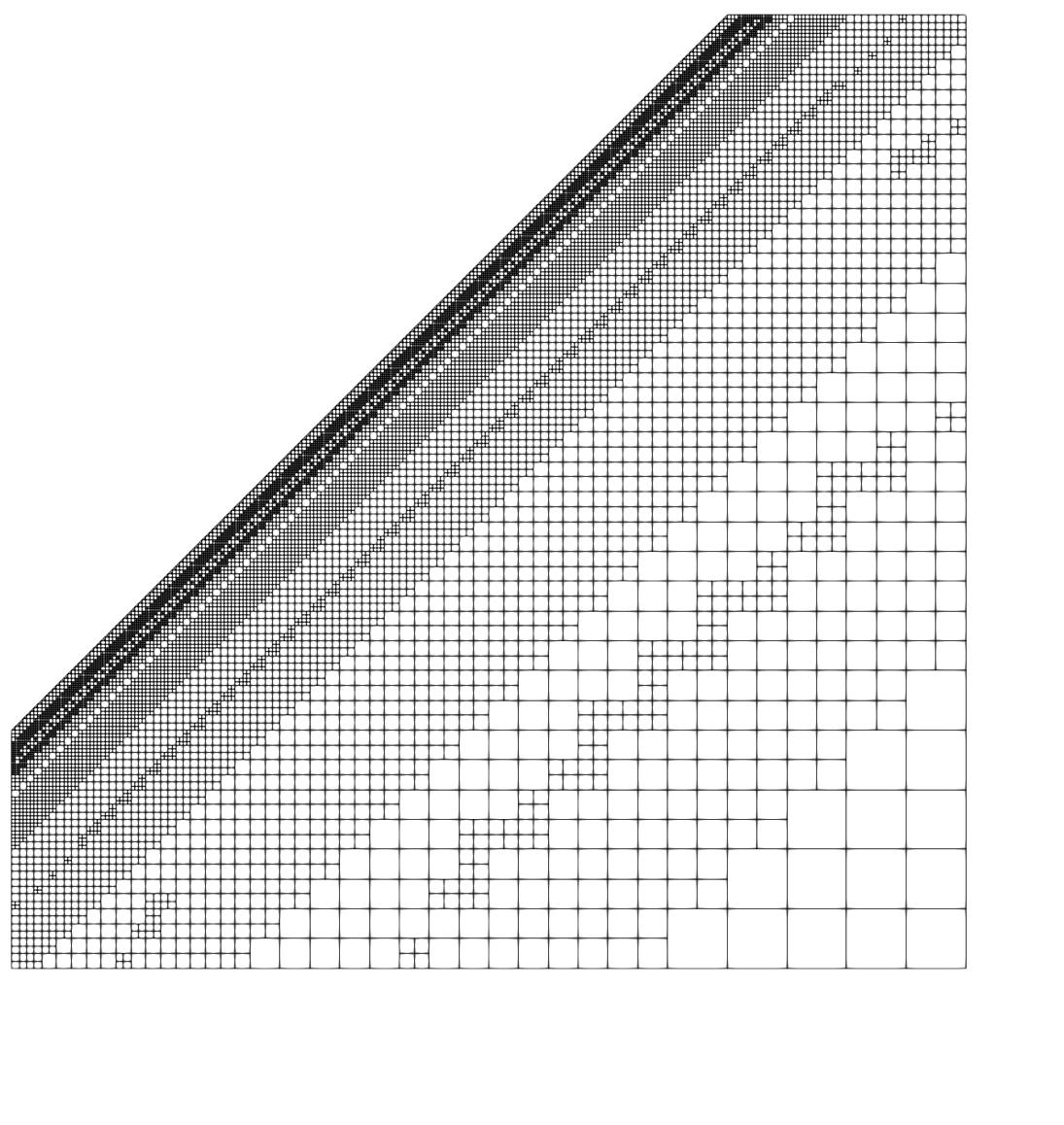}\hspace{2.5cm}
			\caption{\review{Final active mesh obtained with \\adaptive mesh refinement.}} \label{fig:pentagonfinalmesh}
		\end{center}
	\end{subfigure}
	~ 
	\begin{subfigure}[t]{0.5\textwidth}
		\begin{center}
			\begin{tikzpicture}[scale=0.77]
			\begin{axis}[xmode=log, ymode=log, xlabel=$N_\text{dof}$, ylabel=Error/Estimator, legend pos = outer north east, legend style={draw=none,row sep=0.15cm}, grid, grid style={dotted},ymin=3.5e-6,ymax=70,xmin = 15, xmax = 20000]
			\addlegendimage{black, thick, densely dashed};
			\addlegendimage{black, thick};
			\addlegendimage{mark=o, c1, thick, mark options={scale=1.4}};
			\addlegendimage{mark=square, c2, thick, mark options={scale=1.5}};
			\addlegendimage{mark=triangle, c3, thick, mark options={scale=1.5}};
			\addplot[mark=o, c1, thick, densely dashed, mark options={solid,scale=1.4}] table [x=ndofs, y=est, col sep=comma] {data/pentagon_est_p3_delta3_GR.csv};
			\addplot[mark=o, c1, thick, mark options={solid,scale=1.4}] table [x=ndofs, y=err, col sep=comma] {data/pentagon_est_p3_delta3_GR.csv};
			\addplot[mark=square, c2, thick, densely dashed, mark options={solid,scale=1.5}] table [x=ndofs, y=est, col sep=comma] {data/pentagon_est_p3_delta4_GR.csv};
			\addplot[mark=square, c2, thick, mark options={solid,scale=1.5}] table [x=ndofs, y=err, col sep=comma] {data/pentagon_est_p3_delta4_GR.csv};
			\addplot[mark=triangle, c3, thick, densely dashed, mark options={solid,scale=1.5}] table [x=ndofs, y=est, col sep=comma] {data/pentagon_est_p3_delta5_GR.csv};
			\addplot[mark=triangle, c3, thick, mark options={solid,scale=1.5}] table [x=ndofs, y=err, col sep=comma] {data/pentagon_est_p3_delta5_GR.csv};
			\logLogSlopeReverseTriangle{0.77}{0.1}{0.23}{1.5}{black}
			\legend{$\mathscr E\left(u_h^\varepsilon\right)$, $\left|u-u_h^\varepsilon\right|_{1,\Omega}$, $\varepsilon = 10^{-5}$, $\varepsilon = 10^{-6}$, $\varepsilon = 10^{-7}$};
			\end{axis}
			\end{tikzpicture}
			\caption{Convergence of the error and of the estimator under uniform mesh refinement. \hspace{1.8cm}{\color{white}.}} \label{fig:cvpentagonGR}
		\end{center}
	\end{subfigure}
	~ 
	\begin{subfigure}[t]{0.32\textwidth}
		\begin{center}
			\begin{tikzpicture}[scale=0.77]
			\begin{axis}[xmode=log, ymode=log, xlabel=$N_\text{dof}$, ylabel=Error/Estimator, 
			ylabel near ticks, yticklabel pos=right, grid, grid style={dotted},ymin=3.5e-6,ymax=70, xmin = 15, xmax = 20000]
			\addplot[mark=o, c1, thick,  densely dashed, mark options={solid,scale=1.4}] table [x=ndofs, y=est, col sep=comma] {data/pentagon_est_p3_delta3_GERS01_Tadm.csv};
			\addplot[mark=o, c1, thick, mark options={solid,scale=1.4}] table [x=ndofs, y=err, col sep=comma] {data/pentagon_est_p3_delta3_GERS01_Tadm.csv};
			\addplot[mark=square, c2, thick,  densely dashed, mark options={solid,scale=1.5}] table [x=ndofs, y=est, col sep=comma] {data/pentagon_est_p3_delta4_GERS01_Tadm.csv};
			\addplot[mark=square, c2, thick, mark options={solid,scale=1.5}] table [x=ndofs, y=err, col sep=comma] {data/pentagon_est_p3_delta4_GERS01_Tadm.csv};
			\addplot[mark=triangle, c3, thick,  densely dashed, mark options={solid,scale=1.5}] table [x=ndofs, y=est, col sep=comma] {data/pentagon_est_p3_delta5_GERS01_Tadm.csv};
			\addplot[mark=triangle, c3, thick, mark options={solid,scale=1.5}] table [x=ndofs, y=err, col sep=comma] {data/pentagon_est_p3_delta5_GERS01_Tadm.csv};
			\logLogSlopeReverseTriangle{0.75}{0.1}{0.1}{1.5}{black}
			\end{axis}
			\end{tikzpicture}
			\caption{Convergence of the error and of the estimator under adaptive mesh refinement. } \label{fig:cvpentagonGERS}
		\end{center}
	\end{subfigure}
	\caption{Trimmed pentagon geometry and convergence of the energy error and estimator with respect to the number of degrees of freedom.} \label{fig:pentagon}
\end{figure}
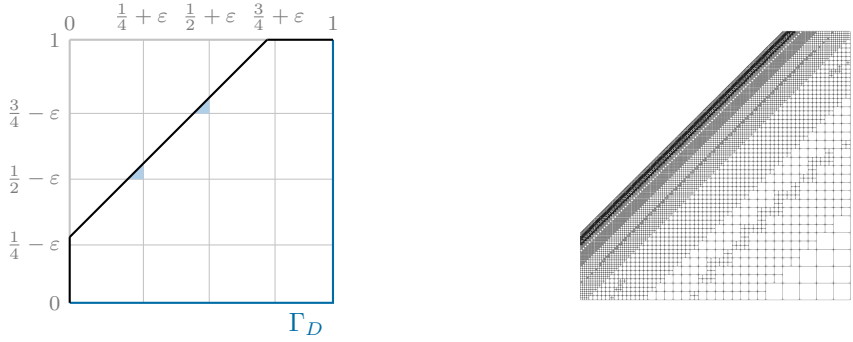
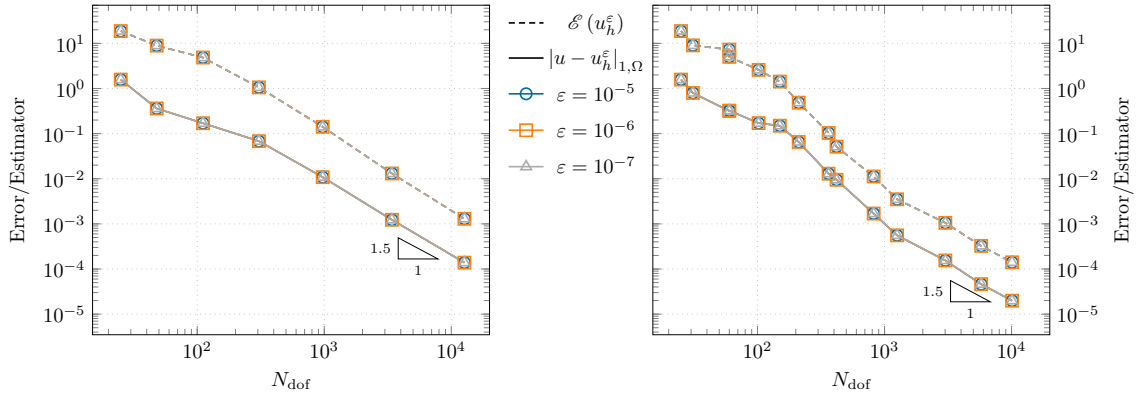

As in the first example, we can see that under both uniform and adaptive refinement, the estimator follows well the behavior of the energy error in the trimmed geometry, both of them having a convergence rate of $\mathcal O\left(N_\text{dof}^{-\frac{p}{2}}\right) = \mathcal O\left(N_\text{dof}^{-\frac{3}{2}}\right)$. Moreover, the presented results confirm the fact that the effectivity index is independent \review{of} $\varepsilon$, that is, it is independent \review{of} the measure of the active part of the trimmed elements. Indeed, for the different chosen values of $\varepsilon$, the curves representing the estimator are almost superposed, and this is also the case for the curves representing the numerical error. At the first iteration, where $\varepsilon \ll h$, the effectivity index is equal to $11.9$ in all cases. Once the asymptotic regime is attained, the effectivity index is equal to $11.0$ when uniform refinement is performed, while it is equal to $7.3$ when adaptive refinement is performed, and this is true for every value of $\varepsilon$.

\subsection{Independence from the size of the active parts of the trimmed elements: singular solution}
To perform a more severe test with respect to the previous one, let us consider a problem whose solution presents a corner singularity, on a trimmed geometry with small active trimmed elements around that corner. To do so, let us consider the L-shaped domain $\Omega$ obtained by trimming the square $(0.5,1)\times (0,0.5)$ from the unit square $(0,1)^2$, as illustrated in Figure \ref{fig:Lshapegeom}. We choose the data $f$, $g_D$ and $g_N$ such that the exact solution of Laplace problem (\ref{eq:originalpb}) is $u(r,\varphi) = r^\frac{2}{3}\sin\left(\displaystyle\frac{2\varphi}{3}\right)$ in the polar coordinate system $(\mathbf e_r, \mathbf e_\varphi)$ centered in $(0.5,0.5)$, as represented in Figure \ref{fig:Lshapegeom}, with $\Gamma_D := \big( (0,1)\times \{1\} \big) \cup \big( \{0\}\times (0,1)\big)$ and $\Gamma_N := \partial \Omega \setminus \overline{\Gamma_D}$. 

\begin{figure}
	\begin{subfigure}[t]{0.48\textwidth}
		\begin{center}
			\begin{tikzpicture}[scale=3.5]
			\draw[c5!50] (0.22,0) -- (0.22,1);
			\draw[c5!50] (0.47,0) -- (0.47,1);
			\draw[c5!50] (0.72,0) -- (0.72,1);
			\draw[c5!50] (0,0.28) -- (1,0.28);
			\draw[c5!50] (0,0.53) -- (1,0.53);
			\draw[c5!50] (0,0.78) -- (1,0.78);
			\draw[-{Stealth[width=2mm]},c3] (0.5,0.5) -- (0.7,0.5); 
			\draw[c3!190] (0.6,0.51) node[below]{$\mathbf{e}_r$}; 
			\draw[-{Stealth[width=2mm]},c3] (0.575,0.5) arc (0:270:0.075);
			\draw[c3!190] (0.38,0.5) node[above]{$\mathbf{e}_\varphi$}; 
			\draw[thick,c1] (0,0) -- (0,1) -- (1,1);
			\draw[thick] (0,0) --(0.5,0) -- (0.5,0.5) -- (1,0.5) -- (1,1);
			\draw[thick,c5!50] (1,0) -- (0.5,0);
			\draw[thick,c5!50] (1,0) -- (1,0.5);
			\draw[thick,c1] (0.1,1) node[above]{$\Gamma_D$};
			\draw[c5!50!black] (1,1) node[right]{\footnotesize $1$};
			\draw[c5!50!black] (1,0.78) node[right]{\footnotesize $\frac{3}{4}+\varepsilon$};
			\draw[c5!50!black] (1,0.53) node[right]{\footnotesize $\frac{1}{2}+\varepsilon$};
			\draw[c5!50!black] (1,0.28) node[right]{\footnotesize $\frac{1}{4}+\varepsilon$};
			\draw[c5!50!black] (1,0) node[right]{\footnotesize $0$};
			\draw[c5!50!black] (1,0) node[below]{\footnotesize $1$};
			\draw[c5!50!black] (0.73,0) node[below]{\footnotesize $\frac{3}{4}-\varepsilon$};
			\draw[c5!50!black] (0.47,0) node[below]{\footnotesize $\frac{1}{2}-\varepsilon$};
			\draw[c5!50!black] (0.21,0) node[below]{\footnotesize $\frac{1}{4}-\varepsilon$};
			\draw[c5!50!black] (0,0) node[below]{\small $0$};
			\end{tikzpicture}
			\caption{Parametric grimmed L-shaped domain $\Omega$ and original mesh (in gray).} \label{fig:Lshapegeom}
		\end{center}
	\end{subfigure}
	~
	\begin{subfigure}[t]{0.48\textwidth}
		\begin{center}
			\begin{tikzpicture}[]
			\pgftext{\includegraphics[scale=0.1]{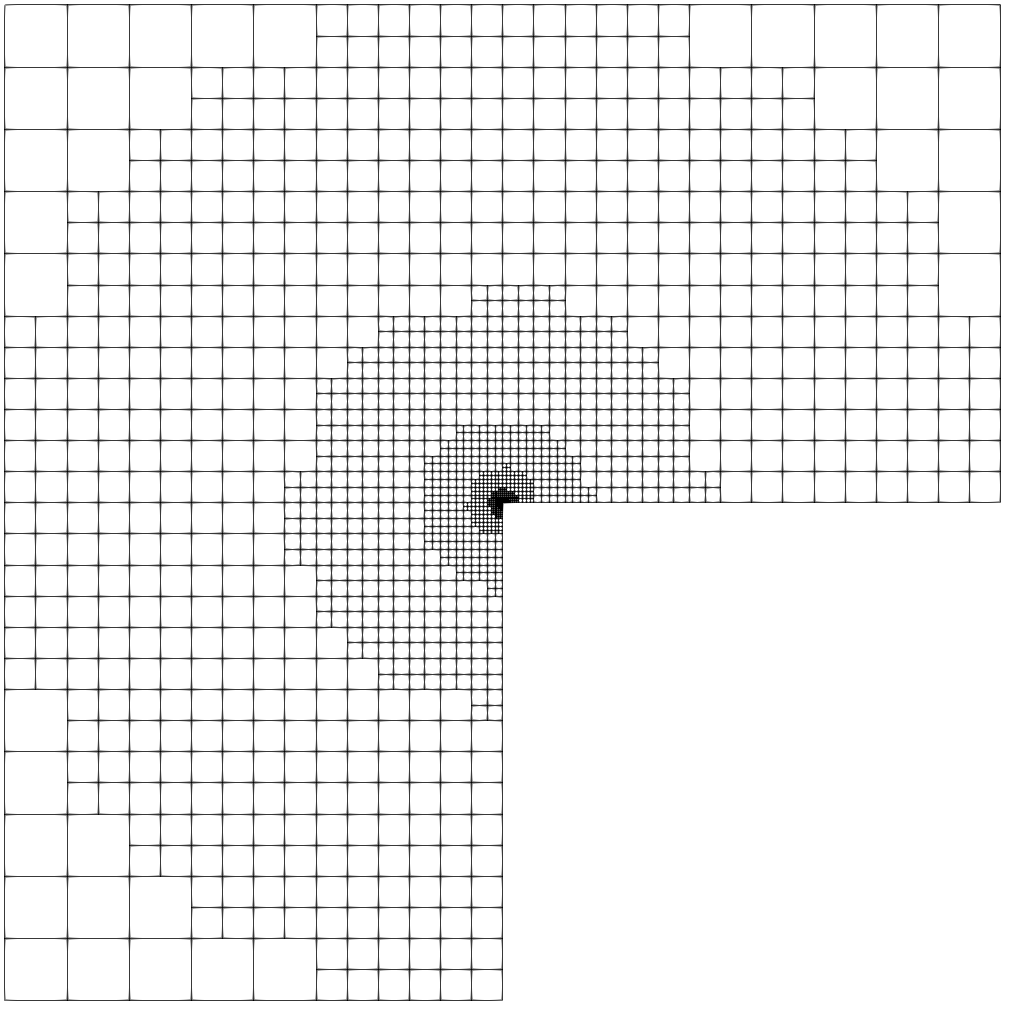}} at (0pt,0pt);
			\draw[white] (0,-2.45) -- (0.1,-2.45);
			\end{tikzpicture}
			\caption{\review{Final active mesh on $\Omega$ obtained with adaptive mesh refinement.}} \label{fig:Lshapefinalmesh}
		\end{center}
	\end{subfigure}
	~ 
	\begin{subfigure}[t]{0.48\textwidth}
		\begin{center}
			\begin{tikzpicture}[scale=1.5]
			\draw[c5!50] (-1.4413, 0.5970) arc (157.5:202.5:1.56);
			\draw[c5!50] (-1.9032, 0.7883) arc (157.5:202.5:2.06);
			\draw[c5!50] (-2.3651, 0.9797) arc (157.5:202.5:2.56);
			\draw[c5!50] (-0.9997, 0.0236) -- (-2.9992, 0.0707);
			\draw[c5!50] (-0.9759, 0.2181) -- (-2.9278, 0.6544);
			\draw[c5!50] (-0.9851, -0.1719) -- (-2.9553, -0.5158);
			\draw[-{Stealth[width=2mm]},c3] (-2,0) -- (-1.6,0); 
			\draw[c3!190] (-1.77,0.02) node[below]{$\mathbf{e}_r$}; 
			\draw[-{Stealth[width=2mm]},c3] (-1.85,0) arc (0:270:0.15);
			\draw[c3!190] (-2.25,0) node[above]{$\mathbf{e}_\varphi$}; 
			\draw[thick,c1] (-2.7716, 1.1481) arc (157.5:202.5:3);
			\draw[thick] (-0.9239,0.3827) arc (157.5:180:1);
			\draw[thick,c5!50] (-1,0) arc (180:202.5:1);
			\draw[thick] (-2,0) arc (180:202.5:2);
			\draw[thick] (-2,0) -- (-1,0);
			\draw[thick,c1] (-2.7716, 1.1481) -- (-0.9239,0.3827);
			\draw[thick,c5!50] (-1.8478,-0.7654) -- (-0.9239,-0.3827);
			\draw[thick] (-2.7716, -1.1481) -- (-1.8478,-0.7654);
			\draw[thick,c1] (-2.8, 1.1481) node[left]{$\Gamma_D$};
			\end{tikzpicture}
			\caption{Trimmed mapped L-shaped domain $\Omega_\mathbf F$ and mapped original mesh (in gray).} \label{fig:curvedLshapegeom}
		\end{center}
	\end{subfigure}
	~ 
	\begin{subfigure}[t]{0.48\textwidth}
		\begin{center}
			\begin{tikzpicture}[]
			\pgftext{\includegraphics[scale=0.09]{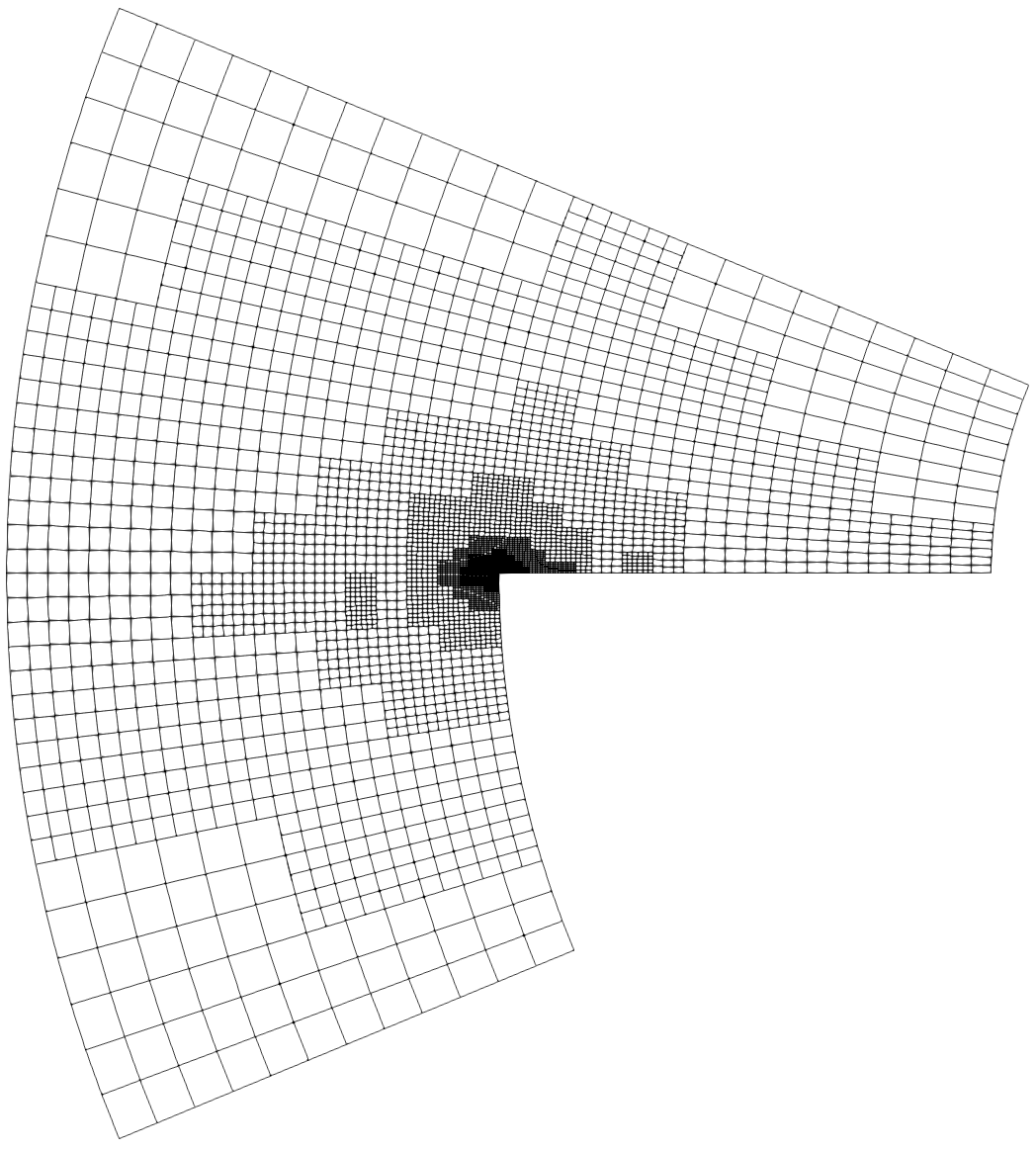}} at (0pt,0pt);
			\end{tikzpicture}
			\caption{Final active mesh on $\Omega_\mathbf F$ obtained with adaptive mesh refinement.} \label{fig:mappedLshapefinalmesh}
		\end{center}
	\end{subfigure}
	\caption{\review{Parametric and mapped L-shaped domains, with their corresponding final active meshes obtained with adaptive mesh refinement.}} \label{fig:Lshapegeommesh}
\end{figure}
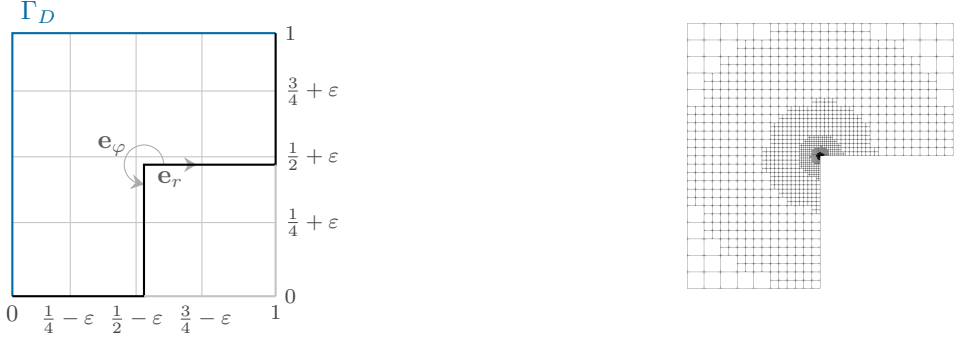
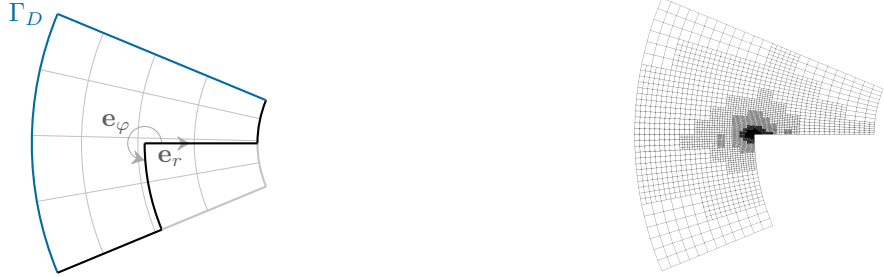

\begin{figure}[h!]
	\begin{subfigure}[tb]{0.62\textwidth}
		\begin{center}
			\hspace*{-2em}\begin{tikzpicture}[scale=0.77]
			\begin{axis}[xmode=log, ymode=log, xlabel=$N_\text{dof}$, ylabel=Error/Estimator, 
			legend style={at={(1.5,1)},anchor=north east,draw=none,row sep=0.15cm},
			grid, grid style={dotted}, xmin=9,xmax=20000,ymax=2.5,ymin=0.001]
			\addlegendimage{black, thick, densely dashed};
			\addlegendimage{black, thick};
			\addlegendimage{mark=o, c1, thick, mark options={scale=1.4}};
			\addlegendimage{mark=square, c2, thick, mark options={scale=1.5}};
			\addlegendimage{mark=triangle, c3, thick, mark options={scale=1.5}};
			\addplot[mark=o, c1, thick, densely dashed, mark options={solid,scale=1.4}] table [x=ndofs, y=est, col sep=comma] {data/Lshape_est2_p2_eps3_GR.csv};
			\addplot[mark=o, c1, thick, mark options={solid,scale=1.4}] table [x=ndofs, y=err, col sep=comma] {data/Lshape_est2_p2_eps3_GR.csv};
			\addplot[mark=square, c2, thick, densely dashed, mark options={solid,scale=1.5}] table [x=ndofs, y=est, col sep=comma] {data/Lshape_est2_p2_eps4_GR.csv};
			\addplot[mark=square, c2, thick, mark options={solid,scale=1.5}] table [x=ndofs, y=err, col sep=comma] {data/Lshape_est2_p2_eps4_GR.csv};
			\addplot[mark=triangle, c3, thick, densely dashed, mark options={solid,scale=1.5}] table [x=ndofs, y=est, col sep=comma] {data/Lshape_est2_p2_eps5_GR.csv};
			\addplot[mark=triangle, c3, thick, mark options={solid,scale=1.5}] table [x=ndofs, y=err, col sep=comma] {data/Lshape_est2_p2_eps5_GR.csv};
			\logLogSlopeReverseTriangle{0.8}{0.1}{0.35}{0.33}{black}
			\legend{$\mathscr E\big(u_h^\varepsilon\big)$, $\left|u-u_h^\varepsilon\right|_{1,\Omega}$, $\varepsilon = 10^{-5}$, $\varepsilon = 10^{-6}$, $\varepsilon = 10^{-7}$};
			\end{axis}
			\end{tikzpicture}
			\caption{Convergence of the error and estimator\\ under uniform mesh refinement in the para-\\metric L-shaped domain $\Omega$.} \label{fig:cvLshapeGR}
		\end{center}
	\end{subfigure}
	~ 
	\begin{subfigure}[tb]{0.38\textwidth}
		\begin{center}
			\hspace*{-2em}\begin{tikzpicture}[scale=0.77]
			\begin{axis}[xmode=log, ymode=log, xlabel=$N_\text{dof}$, ylabel=Error/Estimator, 
			ylabel near ticks, yticklabel pos=right, grid, grid style={dotted},xmin=9,xmax=20000,ymax=2.5,ymin=0.001]
			\addplot[mark=o, c1, thick,  densely dashed, mark options={solid,scale=1.4}] table [x=ndofs, y=est, col sep=comma] {data/Lshape_est2_p2_eps3_GERS01_Tadm.csv};
			\addplot[mark=o, c1, thick, mark options={solid,scale=1.4}] table [x=ndofs, y=err, col sep=comma] {data/Lshape_est2_p2_eps3_GERS01_Tadm.csv};
			\addplot[mark=square, c2, thick,  densely dashed, mark options={solid,scale=1.5}] table [x=ndofs, y=est, col sep=comma] {data/Lshape_est2_p2_eps4_GERS01_Tadm.csv};
			\addplot[mark=square, c2, thick, mark options={solid,scale=1.5}] table [x=ndofs, y=err, col sep=comma] {data/Lshape_est2_p2_eps4_GERS01_Tadm.csv};
			\addplot[mark=triangle, c3, thick,  densely dashed, mark options={solid,scale=1.5}] table [x=ndofs, y=est, col sep=comma] {data/Lshape_est2_p2_eps5_GERS01_Tadm.csv};
			\addplot[mark=triangle, c3, thick, mark options={solid,scale=1.5}] table [x=ndofs, y=err, col sep=comma] {data/Lshape_est2_p2_eps5_GERS01_Tadm.csv};
			\logLogSlopeReverseTriangle{0.52}{0.1}{0.12}{1}{black}
			\end{axis}
			\end{tikzpicture}
			\caption{Convergence of the error and estimator under adaptive mesh refinement in the parametric L-shaped domain $\Omega$.} \label{fig:cvLshapeGERS}
		\end{center}
	\end{subfigure}
	~ 
	\begin{subfigure}[tb]{0.62\textwidth}
		\begin{center}
			\hspace*{-0.5em}\begin{tikzpicture}[scale=0.77]
			\begin{axis}[xmode=log, ymode=log, xlabel=$N_\text{dof}$, ylabel=Error/Estimator, 
			legend style={at={(1.5,1)},anchor=north east,draw=none,row sep=0.15cm},
			grid, grid style={dotted}, xmin=9,xmax=20000,ymax=2.5,ymin=0.001]
			\addlegendimage{black, thick, densely dashed};
			\addlegendimage{black, thick};
			\addlegendimage{mark=o, c1, thick, mark options={scale=1.4}};
			\addlegendimage{mark=square, c2, thick, mark options={scale=1.5}};
			\addlegendimage{mark=triangle, c3, thick, mark options={scale=1.5}};
			\addplot[mark=o, c1, thick, densely dashed, mark options={solid,scale=1.4}] table [x=ndofs, y=est, col sep=comma] {data/curvedLshape_p3_eps3_GR.csv};
			\addplot[mark=o, c1, thick, mark options={solid,scale=1.4}] table [x=ndofs, y=err, col sep=comma] {data/curvedLshape_p3_eps3_GR.csv};
			\addplot[mark=square, c2, thick, densely dashed, mark options={solid,scale=1.5}] table [x=ndofs, y=est, col sep=comma] {data/curvedLshape_p3_eps4_GR.csv};
			\addplot[mark=square, c2, thick, mark options={solid,scale=1.5}] table [x=ndofs, y=err, col sep=comma] {data/curvedLshape_p3_eps4_GR.csv};
			\addplot[mark=triangle, c3, thick, densely dashed, mark options={solid,scale=1.5}] table [x=ndofs, y=est, col sep=comma] {data/curvedLshape_p3_eps5_GR.csv};
			\addplot[mark=triangle, c3, thick, mark options={solid,scale=1.5}] table [x=ndofs, y=err, col sep=comma] {data/curvedLshape_p3_eps5_GR.csv};
			\logLogSlopeReverseTriangle{0.8}{0.1}{0.35}{0.33}{black}
			\legend{$\mathscr E\big(u_h^\varepsilon\big)$, $\left|u-u_h^\varepsilon\right|_{1,\Omega_\mathbf F}$, $\varepsilon = 10^{-5}$, $\varepsilon = 10^{-6}$, $\varepsilon = 10^{-7}$};
			\end{axis}
			\end{tikzpicture}
			\caption{Convergence of the error and estimator \\under uniform mesh refinement in the mapped\\ L-shaped domain $\Omega_\mathbf F$.} \label{fig:curvedcvLshapeGR}
		\end{center}
	\end{subfigure}
	~ 
	\begin{subfigure}[tb]{0.38\textwidth}
		\begin{center}
			\hspace*{-2em}\begin{tikzpicture}[scale=0.77]
			\begin{axis}[xmode=log, ymode=log, xlabel=$N_\text{dof}$, ylabel=Error/Estimator, 
			ylabel near ticks, yticklabel pos=right, grid, grid style={dotted},xmin=9,xmax=20000,ymax=2.5,ymin=0.001]
			\addplot[mark=o, c1, thick,  densely dashed, mark options={solid,scale=1.4}] table [x=ndofs, y=est, col sep=comma] {data/curvedLshape_p3_eps3_GERS01_Tadm.csv};
			\addplot[mark=o, c1, thick, mark options={solid,scale=1.4}] table [x=ndofs, y=err, col sep=comma] {data/curvedLshape_p3_eps3_GERS01_Tadm.csv};
			\addplot[mark=square, c2, thick,  densely dashed, mark options={solid,scale=1.5}] table [x=ndofs, y=est, col sep=comma] {data/curvedLshape_p3_eps4_GERS01_Tadm.csv};
			\addplot[mark=square, c2, thick, mark options={solid,scale=1.5}] table [x=ndofs, y=err, col sep=comma] {data/curvedLshape_p3_eps4_GERS01_Tadm.csv};
			\addplot[mark=triangle, c3, thick,  densely dashed, mark options={solid,scale=1.5}] table [x=ndofs, y=est, col sep=comma] {data/curvedLshape_p3_eps5_GERS01_Tadm.csv};
			\addplot[mark=triangle, c3, thick, mark options={solid,scale=1.5}] table [x=ndofs, y=err, col sep=comma] {data/curvedLshape_p3_eps5_GERS01_Tadm.csv};
			\logLogSlopeReverseTriangle{0.4}{0.1}{0.17}{1.5}{black}
			\end{axis}
			\end{tikzpicture}
			\caption{Convergence of the error and estimator under adaptive mesh refinement in the mapped L-shaped domain $\Omega_\mathbf F$.} \label{fig:curvedcvLshapeGERS}
		\end{center}
	\end{subfigure}
	\caption{\review{Convergence of the energy error and estimator with respect to the number of degrees of freedom in the parametric and mapped L-shaped domains, under uniform and adaptive mesh refinements.}} \label{fig:Lshapecv}
\end{figure}
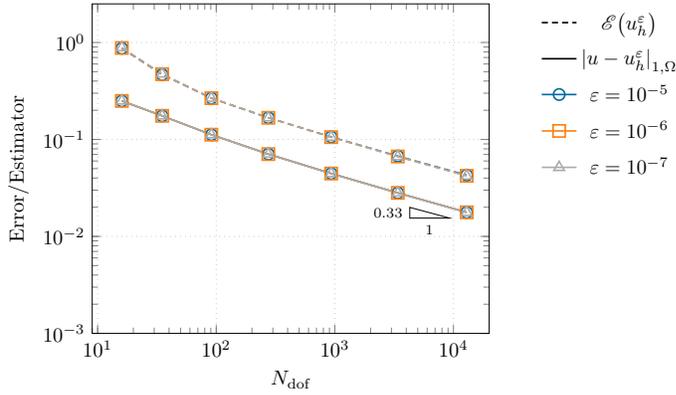
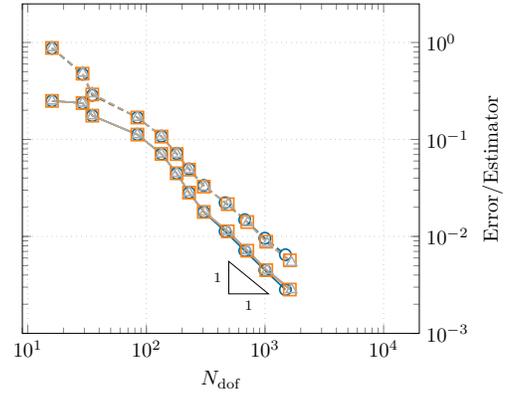
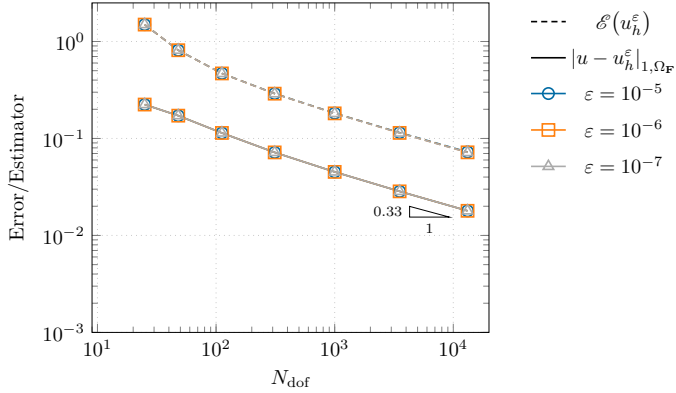
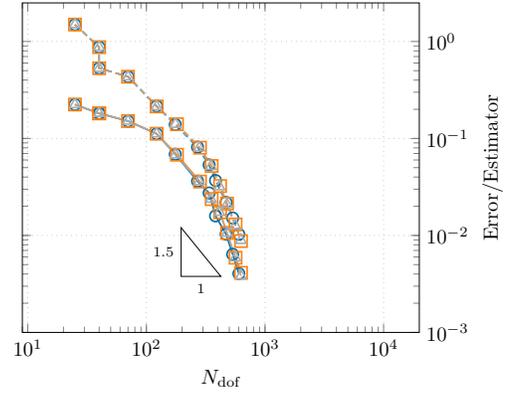

At the first iteration of the adaptive process, the non-trimmed geometry $\Omega_0$ is meshed with $4$ elements in each direction such that the internal knot lines are the lines defined by $x = \displaystyle\frac{k}{4}-\varepsilon$ and $y = \displaystyle\frac{k}{4}+\varepsilon$ for $k=1,2,3$ and $\varepsilon > 0$ small. In this way, the active part of the trimmed mesh elements are very thin and of measure proportional to $\varepsilon h$, as shown in Figure \ref{fig:Lshapegeom}. We solve Galerkin problem (\ref{eq:discrpb}) with HB-splines of degree $2$ in each direction, with $\varepsilon = 10^{-q}$ for $q=5,6,7$, and we denote $u_h^\varepsilon$ the obtained discrete solution. Note that the dependence on $\varepsilon$ is only on the mesh and not on the geometry $\Omega$ itself. As in the previous numerical experiment, both uniform refinement and the adaptive refinement described in Section \ref{sec:adapt} with $\theta = 0.9$ are performed, and the algorithm stops when the number of degrees of freedom of the HB-spline space exceeds $10^4$ or when the number of hierarchical levels exceeds $12$. The results are given in Figure \ref{fig:cvLshapeGR} and \ref{fig:cvLshapeGERS}, respectively. 

As in the previous example, under both uniform and adaptive refinement, the estimator follows well the behavior of the energy error in the trimmed geometry, and no dependence on $\varepsilon$ is observed. Indeed, for the different values of $\varepsilon$, the corresponding curves are almost superposed. Moreover, the chosen differential problem presents a singularity at the corner $(0.5,0.5)$. In particular, it can be shown that the exact solution $u\in H^{\beta-\delta}(\Omega)$ with $\beta = \displaystyle\frac{5}{3}$, for every $\delta > 0$. Therefore, one expects a convergence rate of $\mathcal O\left(h^\frac{2}{3}\right) = \mathcal O\left(N_\text{dof}^{-\frac{1}{3}}\right)$ for uniform $h$-refinement. This is indeed what is observed both for the estimator and the error, with a small effectivity index nearly equal to $2.4$ and independent \review{of} $\varepsilon$. As it is now standard in adaptivity, we can see in Figure \ref{fig:cvLshapeGERS} that the optimal asymptotic convergence rate of $\mathcal O\left(N_\text{dof}^{-\frac{p}{2}}\right) = \mathcal O\left(N_\text{dof}^{-1}\right)$ is recovered when the mesh is adaptively refined. In this case, the effectivity index is very small, being nearly equal to $1.7$ once the asymptotic regime is attained, and it is still independent \review{of} $\varepsilon$. The adaptive strategy exhibits a substantial increase in accuracy with respect to the number of degrees of freedom, in comparison with the uniform refinement strategy. 

Therefore, this test validates the theory presented in Section \ref{sec:estimator}, and in particular, it validates the fact that the estimator is independent \review{of} the type of cut from which the trimmed geometry is obtained. 

\subsection{Singular solution in a mapped trimmed domain}
In this last numerical experiment, we consider the geometry $\Omega_\mathbf F$ obtained from the geometry $\Omega$ of the previous example, mapped with an isogeometric mapping $\mathbf F$ (see Section \ref{ss:hbs}). This numerical test was presented in \cite{bracco}. More precisely and as represented in Figure \ref{fig:curvedLshapegeom}, $\mathbf F(\Omega_0)$ is the surface ruled between the arcs centered in $(2,0)$ of radius $3$ and $1$ whose angle spans between $\displaystyle\frac{7\pi}{8}$ and $\displaystyle\frac{9\pi}{8}$. The curve that trims $\mathbf F(\Omega_0)$ to obtain $\Omega_\mathbf F$ is the image of the trimming curve defining $\Omega$ in the parametric domain, that is, it is the line $\big((0,1)\times \{0\} \big) \cup \mathcal C$, with $\mathcal C$ being the arc centered in $(2,0)$ of radius $2$ whose angle spans between $\pi$ and $\displaystyle\frac{9\pi}{8}$. Therefore, $\Omega_\mathbf F$ presents a re-entrant corner of angle $\displaystyle\frac{\pi}{2}$ in $(0,0)$.
We choose the data $f$, $g_D$ and $g_N$ such that the exact solution of Laplace problem (\ref{eq:originalpb}) is $u(r,\varphi) = r^\frac{2}{3}\sin\left(\displaystyle\frac{2\varphi}{3}\right)$ in the polar coordinate system $(\mathbf e_r, \mathbf e_\varphi)$ centered in $(0,0)$, as represented in Figure \ref{fig:curvedLshapegeom}, with $\Gamma_D := \mathbf F\big( (0,1)\times \{1\} \big) \cup \mathbf F\big( \{0\}\times (0,1)\big)$ and $\Gamma_N := \partial \Omega \setminus \overline{\Gamma_D}$. 

At the first iteration of the adaptive process, the considered mesh is the same as in the previous parametric L-shaped example, so that the active part of the trimmed mesh elements are very thin and of measure proportional to $\varepsilon h$. We solve Galerkin problem (\ref{eq:discrpb}) with HB-splines of degree $3$ in each direction, with $\varepsilon = 10^{-q}$ for $q=5,6,7$, and we denote $u_h^\varepsilon$ the obtained discrete solution. Again, note that the dependence on $\varepsilon$ is only on the mesh and not on the geometry $\Omega$ itself. As in the previous numerical experiment, both uniform refinement and the adaptive refinement described in Section \ref{sec:adapt} with $\theta = 0.9$ are performed, and the algorithm stops when the number of degrees of freedom of the HB-spline space exceeds $10^4$ or when the number of hierarchical levels exceeds $12$. The results are given in Figure \ref{fig:curvedcvLshapeGR} and \ref{fig:curvedcvLshapeGERS}, respectively. 

As expected, the results are very similar to the ones of the previous experiment. Indeed, under both uniform and adaptive refinement, the estimator follows well the behavior of the energy error in the trimmed mapped geometry, and no dependence on $\varepsilon$ is observed. Moreover, the considered differential problem presenting a corner singularity in $(0,0)$, as previously, one expects a convergence rate of $\mathcal O\left(\review{N_\text{dof}^{-\frac{1}{3}}}\right)$ for uniform refinement, while one expects to recover the optimal asymptotic convergence rate of $\mathcal O\left(N_\text{dof}^{-\frac{p}{2}}\right) = \mathcal O\left(N_\text{dof}^{-\frac{3}{2}}\right)$ for adaptive refinement. This is indeed the case for the uniform refinement, and a small effectivity index nearly equal to $4.1$ and independent \review{of} $\varepsilon$ is obtained. However, a faster convergence than the expected one is observed for the adaptive refinement, both for the error and the estimator: this is a common behavior when the asymptotic regime is not reached yet. Also in this case, a small effectivity index is obtained, being nearly equal to $2.3$ and independent \review{of} $\varepsilon$.

%% file: appendix.tex
\begin{mylemma} \label{lemma:savare}
	For all $q\in [2,\infty)$ and all $v\in H^\frac{1}{2}(0,2\pi)$, 
	$$\|v\|_{L^q(0,2\pi)} \leq c \sqrt{q} \|v\|_{\frac{1}{2}, (0,2\pi)},$$
	where $c$ is a constant independent from $q$. 
\end{mylemma}
\begin{proof}
	See \cite[Lemma~5.1]{benbelgacembuffamaday}.
\end{proof}

\begin{mylemma}\label{lemma:LqH1}
	For all $q\in [2,\infty)$ and all $v\in H^1\left((0,2\pi)^2\right)$, 
	$$\|v\|_{L^q\left((0,2\pi)^2\right)} \leq c \sqrt{q} \|v\|_{1, (0,2\pi)^2},$$
	where $c$ is a constant independent from $q$. 
\end{mylemma}
\begin{proof}
	The proof follows \cite[Lemma~5.1]{benbelgacembuffamaday}, but in the two-dimensional case.
	In order to use the Fourier decomposition, consider the Hilbert basis of $L^2\left((0,2\pi)^2\right)$ defined by $\varphi_{\mathbf k}(x) := \varphi_{k_1}(x_1)\varphi_{k_2}(x_2)$ for all $\mathbf k = (k_1,k_2) \in \mathbb N^2$, $x=(x_1,x_2)\in (0,2\pi)^2$, where for $i=1,2$, $\varphi_0(x_i) = \frac{1}{\sqrt{2\pi}}$ and $\varphi_k(x_i) = \frac{1}{\sqrt{\pi}}\cos(kx_i)$, $k\in\mathbb N\setminus\{0\}$. 
	Let $v \in H^1\left((0,2\pi)^2\right)$ be decomposed as $v = \displaystyle\sum_{\mathbf k\in\mathbb N^2}v_{\mathbf k} \varphi_{\mathbf k}$. Then 
	\begin{equation} \label{eq:H1fourier}
		\|v\|_{1,(0,2\pi)^2} = \left[\sum_{\mathbf k\in\mathbb N^2} (1+k_1^2+k_2^2) v_{\mathbf k}^2 \right]^\frac{1}{2}.
	\end{equation}
	Now, for every $q\in[2,\infty)$, let $q'\in (1,2]$ such that $\displaystyle\frac{1}{q}+\frac{1}{q'} = 1$. Then by the Hausdorff-Young inequality, there exists a constant $C$ independent from $q$ such that
	\begin{equation*}
		\|v\|_{L^q\left((0,2\pi)^2\right)} \leq C \left( \sum_{\mathbf k\in\mathbb N^2} \left|v_\mathbf k\right|^{q'} \right)^\frac{1}{q'} = C \left( \sum_{\mathbf k\in\mathbb N^2} \left(1+k_1^2+k_2^2\right)^{-\frac{q'}{2}} \left(1+k_1^2+k_2^2\right)^\frac{q'}{2} \left|v_\mathbf k\right|^{q'} \right)^\frac{1}{q'}.
	\end{equation*}
	Thus if we let $r\in (2,\infty]$ such that $\displaystyle\frac{1}{r}+\frac{q'}{2} = 1$, by H\"older inequality and using (\ref{eq:H1fourier}), 
	\begin{align*}
		\|v\|_{L^q\left((0,2\pi)^2\right)} &\leq C \left[ \sum_{\mathbf k\in\mathbb N^2}\left(1+k_1^2+k_2^2\right)^{-\frac{q'r}{2}} \right]^\frac{1}{q'r} \left[ \sum_{\mathbf k\in\mathbb N^2}\left(1+k_1^2+k_2^2\right) v_\mathbf k^2 \right]^\frac{1}{2} \\
		&= C \left[ \sum_{\mathbf k\in\mathbb N^2}\left(1+k_1^2+k_2^2\right)^{-\frac{q}{q-2}} \right]^\frac{q-2}{2q} \|v\|_{1,(0,2\pi)^2}.
	\end{align*}
	One can conclude the proof by observing that the Riemann serial is bounded by $\sqrt q$. 
\end{proof}